\documentclass[11pt]{amsart}
\usepackage{amsmath,amssymb,mathtools}
\usepackage{tikz-cd}
\usepackage[margin=20mm]{geometry}
\usepackage[T1]{fontenc}
\usepackage{xcolor}
\usepackage[colorlinks=true,linkcolor=blue,citecolor=blue,urlcolor=blue]{hyperref}
\newtheorem{thm}{Theorem}[section]
\newtheorem{prop}[thm]{Proposition}
\newtheorem{lem}[thm]{Lemma}
\newtheorem{cor}[thm]{Corollary}
\newtheorem{alphthm}{Theorem}

\theoremstyle{definition}
\newtheorem{defn}[thm]{Definition}
\newtheorem{setup}[thm]{Setup}
\newtheorem{notation}[thm]{Notation}
\theoremstyle{remark}
\newtheorem{rem}[thm]{Remark}
\newcommand{\C}{\mathbb C}
\newcommand{\Q}{\mathbb Q}
\newcommand{\R}{\mathbb R}
\newcommand{\Pj}{\mathbb P}
\newcommand{\cO}{\mathcal O}
\newcommand{\cJ}{\mathcal J}
\newcommand{\frakj}{\mathfrak j}
\newcommand{\eps}{\epsilon}

\DeclareMathOperator{\Supp}{Supp}
\DeclareMathOperator{\Exc}{Exc}
\DeclareMathOperator{\ord}{ord}
\DeclareMathOperator{\vol}{vol}
\DeclareMathOperator{\Spec}{Spec}
\DeclareMathOperator{\Bl}{Bl}
\DeclareMathOperator{\Frac}{Frac}
\DeclareMathOperator{\trdeg}{trdeg}
\DeclareMathOperator{\lct}{lct}
\title{Boundedness of minimal modifications}
\author{Junpeng Jiao}
\address{Yau Mathematical Sciences Center, Tsinghua University}
\date{}
\begin{document}
\begin{abstract}
For fixed $n$ and $\eps>0$, we prove boundedness of projective birational morphisms $f\colon X\to\Pj^n$ over $\C$ for which $(X,B)$ is an $\eps$-lc pair, $B\ge0$, $K_X+B$ is nef over $\Pj^n$, and $\deg(f_*B)$ is bounded. The nonzero coefficients of $B$ need not have a positive lower bound, and $(\Pj^n,f_*B)$ need not be log canonical. The $f$-exceptional prime divisors on $X$, equipped with their reduced induced scheme structures, form a bounded family of projective schemes, including when they are nonnormal. We also obtain a uniform positive lower bound for the log canonical thresholds of pullbacks of all effective target divisors of bounded degree. The boundedness statement proved in this paper was originally conjectured by Caucher Birkar.
\end{abstract}
\maketitle

\section{Introduction}\label{sec:introduction}

Throughout the paper, the ground field is $\C$, and $\Pj^n=\Pj^n_\C$.

A birational morphism can carry exceptional divisors of increasingly complicated geometry even when its target is fixed. A relatively nef log canonical divisor restricts the modifications that can occur. The question considered here is whether this restriction remains effective when the positive boundary coefficients have no lower bound. We use the degree of the pushforward boundary, which takes account of its coefficients.

The minimal model program supplies existence and extraction results for klt pairs, while boundedness theorems for polarized pairs apply when the nonzero boundary coefficients have a fixed positive lower bound; see \cite{BCHM10,Fujino2011,BirkarPolarised}. The issue in the present setting is to construct a boundary with such a lower bound on the given variety $X$. We also seek boundedness of the reduced exceptional divisors as closed subschemes of $X$, which requires retaining their scheme structures, including possible nonnormality.

Birkar's boundedness theorem for Fano type fibrations includes boundedness of the varieties underlying effective crepant models over bounded bases, with bounded weighted boundary degree and no bound on the reduced boundary support; see \cite[Theorem~1.2 and the discussion of crepant models in Section~1]{BirkarFibrations24}. Here we allow $K_X+B$ to be merely nef over the target, and the pushforward pair need not be lc. The use of root orders and two ideals in a reduction of dimension is closely related to Musta\c{t}\u{a}'s inductive description of log canonicity for monic polynomials; see \cite[Section~4 and Theorem~5.3]{MustataInductive11}. Proposition~\ref{prop:ideals} constructs global coefficient ideals for a signed boundary and bounds the normalized generating degree of the positive ideal independently of the degree of the negative part, as required for our boundedness argument.

The conjecture was inspired by Birkar's work \emph{Singularities on Fano fibrations and beyond}~\cite{BirkarFibrationsBeyond}. His multiplicity theorem~\cite[Theorem~1.10]{BirkarFibrationsBeyond} bounds the multiplicity of a fiber along a divisor of small log discrepancy on a relatively nef birational model. Its hypotheses include an auxiliary lc boundary $\Lambda$ for which the divisor is an lc place. Birkar expects a version without $\Lambda$ and the assumptions involving it. Theorem~\ref{thm:main} is a global analogue over the fixed target $\Pj^n$: control of the weighted degree of the pushforward boundary gives boundedness of the birational morphisms and uniform discrepancy control, without an auxiliary lc boundary. Corollary~\ref{cor:bounded-divisor-threshold} below also gives uniform thresholds for pullbacks of arbitrary effective target divisors of bounded degree.

The conventions on boundedness, degrees of $\R$-divisors, and log discrepancies are collected in Section~\ref{sec:preliminaries}.

\begin{alphthm}[Boundedness of birational morphisms]\label{thm:main}
Fix positive integers $n,k$ and $0<\eps<1$. Consider projective birational morphisms $f\colon X\to Y=\Pj^n$, where $X$ is normal and integral, together with an effective $\R$-boundary $B$ satisfying
\begin{enumerate}
\item $K_X+B$ is $\R$-Cartier and $(X,B)$ is $\eps$-lc;
\item $K_X+B$ is nef over $Y$;
\item $\deg(f_*B)\le k$.
\end{enumerate}
The morphisms $f$ form a bounded class over the fixed target $Y$. The reduced $f$-exceptional prime divisors on $X$ form a bounded class of projective schemes, with no normality or Cartier assumption on those divisors.

For $n\ge2$ there is a positive rational number $\beta=\beta(n,k,\eps)$ such that
\[
 (X,B+\beta f^*\Delta)\text{ is }\eps/2\text{-lc},\qquad \Delta=f_*B.
\]
Every normalized divisorial valuation $v$ then satisfies
\begin{equation}\label{eq:main-comparison}
 (1+\beta)A(v,X,B)\ge\beta A(v,Y,0)+\eps/2.
\end{equation}
\end{alphthm}

Relative bigness, if included in the hypotheses, is automatic: the generic fiber of a birational morphism is a point. No separate $\Q$-Gorenstein or $\Q$-factorial assumption on $X$ is required. We impose no positive lower bound on the nonzero coefficients, bound on their denominators, bound on Cartier indices, or bound on the degree of the reduced support of $\Delta$.

\begin{cor}[Thresholds of pullbacks of bounded-degree divisors]\label{cor:bounded-divisor-threshold}
Fix positive integers $n,k,\ell$ and $0<\eps<1$. There is $t=t(n,k,\ell,\eps)>0$ such that, for every pair and morphism in Theorem~\ref{thm:main} and every effective $\R$-divisor $D$ on $Y$ with $\deg D\le\ell$,
\[
 (X,B+t f^*D)\text{ is }\eps/2\text{-lc}.
\]
In particular, $\lct(X,B;f^*D)\ge t$. For $n\ge2$, one may take
\[
 t=\frac{\beta(n,k,\eps)}{2(1+\beta(n,k,\eps))\ell},
\]
and for $n=1$ one may take $t=\eps/(2\ell)$.

Equivalently, fix a real number $r\ge1$ and let $A$ be any very ample Cartier divisor on $Y=\Pj^n$ with $A^n\le r$. Then there is a uniform $t=t(n,k,r,\eps)>0$ such that the same conclusion holds for every effective $\R$-divisor $D$ with $A-D$ pseudo-effective. The constant is independent of $A$, $D$, and the coefficients and reduced support of $D$.
\end{cor}

The central estimate bounds $A(E,Y,0)$ for every $f$-exceptional prime divisor $E$ on $X$. Its proof proceeds by induction on the dimension. First, we construct an $\R$-divisor $\Theta=\Delta-N$ on $Y$, where $N\ge0$, such that $A(E,Y,\Theta)=A(E,X,B)$ for these divisors. Projection from a suitable point gives a formula for discrepancies in terms of the valuations of differences of roots. We express the resulting maximum using two coherent ideals and control the degree of generators of the first independently of $\deg N$. After principalizing these ideals, general members of base point free linear systems yield an effective pair in dimension $n-1$ with a uniform positive lower bound on log discrepancies, to which the induction hypothesis applies.

We then apply this discrepancy bound to all morphisms satisfying the hypotheses with $k$ replaced by $2k$ and $\eps$ replaced by $\eps/2$. A divisor computing
\[
 \sup\{t\ge0:(X,B+t f^*\Delta)\text{ is }\eps/2\text{-lc}\}
\]
can be extracted on a variety in this class, giving a uniform positive lower bound for the supremum. Multiplier ideals then yield an effective boundary on $X$ whose nonzero coefficients have a uniform positive lower bound. The polarized boundedness theorem of \cite[Theorem~1.5]{BirkarPolarised} applies to these pairs after adding general pullbacks of hyperplanes. Retaining one such pullback in the bounded family determines $f^*\cO_Y(1)$; choosing bases of its global sections recovers the morphisms to the fixed target $Y$.

\begin{alphthm}[Explicit discrepancy estimates]\label{thm:explicit}
For $n\ge2$, put
\[
 \begin{gathered}
 a_n=3(n-1)2^{n-2},\qquad b_n=3(2^{n-1}-1),\\
 C_n=24^{n-1}12^{3(2^{n-1}-n)},\qquad D_n=6C_n2^{a_n}3^{b_n}.
 \end{gathered}
\]
Under the hypotheses of Theorem~\ref{thm:main}, every $f$-exceptional prime divisor $E$ on $X$ and every normalized divisorial valuation $v$ satisfy
\[
 A(E,Y,0)\le C_n2^{b_n}k^{a_n}\eps^{-b_n},\qquad
 A(v,X,B)\ge D_n^{-1}k^{-a_n-1}\eps^{b_n+1}A(v,Y,0).
\]
\end{alphthm}

These estimates give explicit bounds for log discrepancies. Section~\ref{sec:preliminaries} collects the conventions and proves the discrepancy comparisons and the construction of $\Theta$ used above. Sections~\ref{sec:ramification}--\ref{sec:descent} develop projection and the reduction of dimension. Sections~\ref{sec:induction} and \ref{sec:boundedness} prove Theorem~\ref{thm:main} and Corollary~\ref{cor:bounded-divisor-threshold}. Section~\ref{sec:applications} gives the explicit estimates and examples.

\subsection*{Acknowledgements}
The author is grateful to his postdoctoral mentor, Caucher Birkar, for his guidance and help, and to the Yau Mathematical Sciences Center for providing AI support. The author was supported by the National Key Research and Development Program of China (Grant No.~2025YFA1018100) and by start-up funding from Tsinghua University.

\section{Preliminaries}\label{sec:preliminaries}

We follow Lazarsfeld~\cite{Laz04I,Laz04II} and Koll\'ar--Mori~\cite{KM98} for standard terminology. Our conventions for sub-pairs and bounded couples agree with \cite[Sections~2.7 and~2.13]{BirkarPolarised}. We record the definitions and notation needed below.

\subsection{Valuations, centers, and normalization}\label{subsec:valuation-basics}

We recall the conventions used below; see \cite[Section~1.2]{JM12} and \cite[Tag~00I8]{StacksFoundations}. Let $\mathcal K$ be a function field over $\C$. A \textit{valuation} means an additive real valuation $v\colon \mathcal K^\times\to\R$, trivial on $\C^\times$. We set $v(0)=+\infty$ and require
\[
 v(fg)=v(f)+v(g),\qquad
 v(f+g)\ge\min\{v(f),v(g)\}.
\]
In particular, $v(f/g)=v(f)-v(g)$ for nonzero $f,g$. Its \textit{value group} is $\Gamma_v=v(\mathcal K^\times)\subseteq\R$. The valuation is \textit{trivial} if $v(f)=0$ for every $f\ne0$, and \textit{nontrivial} otherwise.

The \textit{valuation ring}, its maximal ideal, and its \textit{residue field} are
\[
 \cO_v=\{f\in \mathcal K:v(f)\ge0\},\qquad
 \mathfrak m_v=\{f\in \mathcal K:v(f)>0\},\qquad
 \kappa(v)=\cO_v/\mathfrak m_v.
\]
The units of $\cO_v$ are precisely the elements of value zero. A nontrivial valuation is \textit{discrete} if $\Gamma_v=\lambda\mathbb Z$ for some $\lambda>0$. Its ring is then a \textit{discrete valuation ring} (DVR). A generator $\varpi$ of $\mathfrak m_v$ is a \textit{uniformizer}, and $v(\varpi)=\lambda$. Such a valuation is \textit{normalized} if $\lambda=1$, so its value group is exactly $\mathbb Z$.

For a normal variety $V$, write $K(V)$ for its field of rational functions. Suppose $\mathcal K=K(V)$. The \textit{center} of $v$ on $V$, when it exists, is the irreducible closed subvariety $c_V(v)$ whose generic point $\xi$ satisfies
\[
 \cO_{V,\xi}\subseteq\cO_v,\qquad
 \mathfrak m_v\cap\cO_{V,\xi}=\mathfrak m_{V,\xi}.
\]
Thus functions regular at $\xi$ have nonnegative value, and those vanishing there have positive value. We use the closed subvariety, rather than just its generic point, when referring to a center. If $V$ is proper, the center exists and is unique. For an effective Cartier divisor with local equation $h$ at $\xi$, not containing $c_V(v)$ means that $h$ is a unit there, and hence $v(h)=0$.

Let $p\colon W\to V$ be a proper birational morphism with $W$ normal. A prime divisor $E\subset W$ defines the \textit{order of vanishing} $\ord_E$ on $\mathcal K$. At the generic point of $E$, the local ring is a DVR: writing $f=\varpi^m u$ with $u$ a unit gives $\ord_E(f)=m$. A valuation is \textit{divisorial} if it has the form $c\,\ord_E$ for some such $E$ and $c>0$. Its normalized representative is $\ord_E$, whose value group is $\mathbb Z$ and whose center on $V$ is $p(E)$. For example, $2\ord_E$ has value group $2\mathbb Z$ and is not normalized. Having value group $\mathbb Z$ alone does not assert that a valuation is divisorial. Unless a scaling is explicitly specified, all divisorial valuations in this paper are normalized.

For a subfield $\C\subseteq \mathcal K_0\subseteq \mathcal K$, the \textit{restriction} $v|_{\mathcal K_0}$ evaluates the same function $v$ only on elements of $\mathcal K_0$. If $\Gamma_v=\mathbb Z$ and this restriction is nontrivial, its value group is $e\mathbb Z$ for a unique positive integer $e$. Hence
\[
 v|_{\mathcal K_0}=e\zeta,\qquad \zeta(g)=\frac{v(g)}e\quad(g\in \mathcal K_0^\times),
\]
where $\zeta$ is a normalized discrete valuation of $\mathcal K_0$. In projection, $\mathcal K=K(Y)$ and $\mathcal K_0=K(Z)$; Proposition~\ref{prop:ramification} proves the additional assertion that $\zeta$ is divisorial. In contrast, the rescaled valuation $v/e$ is still defined on all of $\mathcal K$, and restricts to $\zeta$ on $\mathcal K_0$.

An \textit{extension} of $\zeta$ from $\mathcal K_0$ to a larger field $\mathcal K'$ is a valuation $\widetilde\zeta$ satisfying $\widetilde\zeta|_{\mathcal K_0}=\zeta$. If $\zeta$ is normalized discrete and $\varpi$ is a uniformizer of $\cO_\zeta$, then $\widetilde\zeta(\varpi)=1$, but $\widetilde\zeta$ need not be normalized on $\mathcal K'$. When $\mathcal K'/\mathcal K_0$ is finite, $\cO_{\widetilde\zeta}$ is again a DVR. If $s$ is its uniformizer and $\varpi=u s^r$ with $u$ a unit, then $r$ is the \textit{ramification index} and $\widetilde\zeta(s)=1/r$. Its normalized representative is $r\widetilde\zeta$. This distinction is used when passing to a splitting field in Notation~\ref{defn:contacts}.

\subsection{Divisors, pairs, and boundedness}

\begin{defn}[Divisors and pairs]\label{defn:signed-degree}\label{defn:discrepancy}
For an $\R$-divisor $D=\sum_i d_iD_i$, write
\[
 D^+=\sum_i\max\{d_i,0\}D_i,\qquad
 D^-=\sum_i\max\{-d_i,0\}D_i,
\]
so $D=D^+-D^-$ and $D^+,D^-$ have no common component. The \textit{support} $\Supp D$ is the union of the prime divisors $D_i$ with $d_i\ne0$, equipped with its reduced structure. Two divisors have \textit{no common component} if no prime divisor occurs with nonzero coefficient in both; their supports may still intersect. On $\Pj^r$, we use the usual degree
\[
 \deg D=D\cdot\cO_{\Pj^r}(1)^{r-1}=\sum_i d_i\deg D_i.
\]
A \textit{sub-pair} $(V,D)$ consists of a normal variety $V$ and an $\R$-divisor $D$ such that $K_V+D$ is $\R$-Cartier. It is a \textit{pair} if $D\ge0$, and $D$ is a \textit{boundary} if its coefficients belong to $[0,1]$. It is \textit{log smooth} if $V$ is smooth and $\Supp D$ has simple normal crossings (SNC).

For a normalized divisorial valuation $v=\ord_E$ as in Subsection~\ref{subsec:valuation-basics}, take a log resolution $p\colon W\to V$ containing $E$ and write
\[
 K_W+D_W=p^*(K_V+D),\qquad
 A(v,V,D)=1-\operatorname{coeff}_E D_W.
\]
Thus $A$ denotes \textit{log discrepancy}, and $K_W+D_W$ is the \textit{crepant pullback} of $K_V+D$. For a rescaled divisorial valuation we set $A(cv,V,D)=cA(v,V,D)$, for $c>0$. For $0\le\eta\le1$, the sub-pair is \textit{sub-$\eta$-lc} if $A(v,V,D)\ge\eta$ for every normalized divisorial valuation $v$. If $(V,D)$ is sub-$\eta$-lc, then for such a $v$ and a positive integer $e$, the rescaled valuation $v/e$ on the same function field satisfies $A(v/e,V,D)=A(v,V,D)/e\ge\eta/e$; the bound $\eta$ applies to the normalized valuation. When $D\ge0$ we omit ``sub''; for $\eta=0$ we write sub-lc or lc. Canonical divisors are chosen compatibly. We use $v(D)$ for the pullback coefficient of an $\R$-Cartier divisor, and $v(\frakj)$ for the minimum of the values of local generators of a nonzero coherent ideal.

For a prime divisor $E$ over $V$, we also write $A(E,V,D):=A(\ord_E,V,D)$. For an $\R$-Cartier divisor $T$ on $V$, we write $\ord_E(T):=\operatorname{coeff}_{E_W}(p^*T)$, where $p\colon W\to V$ is any birational model on which the valuation $\ord_E$ is represented by the prime divisor $E_W$; this value is independent of the chosen model.
\end{defn}

\begin{defn}[Boundedness]\label{defn:bounded}
A collection of projective morphisms $X\to\Pj^n$ is \textit{bounded} if finitely many projective families $\mathcal X_i\to S_i$ of finite type, equipped with morphisms $\mathcal X_i\to\Pj^n\times S_i$, contain every morphism in the collection as a geometric fiber, up to isomorphism over $\Pj^n$. For projective schemes, boundedness has the analogous meaning without a target map. A collection of integral projective schemes is \textit{birationally bounded} if every member is birational to a geometric fiber of finitely many projective families of finite type.

A \textit{couple} $(V,D)$ consists of a normal projective variety and a reduced divisor, without a Cartier condition on $K_V+D$. A collection of couples is bounded if finitely many projective families with reduced divisors contain each couple up to isomorphism, allowing its divisor to be a union of components of the divisor on the fiber. Pairs are \textit{log bounded} when their associated couples $(V,\Supp D)$ are bounded. A family of couples is \textit{log smooth over the base} when the variety is smooth over the base and the divisor has relative simple normal crossings. When specified components or maps to $\Pj^n$ are to be retained, these are included in the families.
\end{defn}

\subsection{Birational models and thresholds}

\begin{setup}\label{setup:source}
Write $Y=\Pj^n$, $H=\cO_Y(1)$, and let $f\colon X\to Y$ satisfy the hypotheses of Theorem~\ref{thm:main}. Put $L=K_X+B$, $\Delta=f_*B$, $d=\deg\Delta$, and
\[
 F=f^*(K_Y+\Delta)-L.
\]
An $f$-exceptional prime divisor always means a prime divisor on $X$ contracted by $f$. When considered as a projective scheme, it carries its reduced closed subscheme structure on $X$.
\end{setup}

\begin{notation}[Divisors on birational models]\label{defn:source-reference}
Let $\rho\colon S\to Y$ be a smooth projective birational model. On a common resolution $p\colon W\to X$, $g\colon W\to S$, put
\[
 K_W+B_W=p^*(K_X+B),\qquad B_S=g_*B_W.
\]
Equivalently, the coefficient of every prime $T\subset S$ in $B_S$ is $1-A(T,X,B)$, so $B_S$ is independent of the common resolution. This notation does not require a morphism between $S$ and $X$ over $Y$.
\end{notation}

\begin{defn}[$\eta$-lc thresholds]\label{defn:margin}
For an $\eta$-lc pair $(V,D)$, $0\le\eta\le1$, and an effective $\R$-Cartier divisor $T$, its \textit{$\eta$-lc threshold} is
\[
 t_\eta(V,D;T)=\sup\{t\ge0:(V,D+tT)\text{ is }\eta\text{-lc}\}.
\]
We allow the value $\infty$ when $T=0$. At $\eta=0$ this is the usual log canonical threshold, denoted by $\lct(V,D;T)$. In Setup~\ref{setup:source}, we apply this notation to $T=f^*\Delta$ and $\eta=\eps/2$. Corollary~\ref{cor:bounded-divisor-threshold} treats pullbacks of arbitrary effective target divisors of bounded degree.
\end{defn}

\subsection{Projection}

\begin{notation}[Projection]\label{defn:directions}
For $n\ge2$, projection from $q\in Y=\Pj^n$ is resolved by the $\Pj^1$-bundle $\Bl_qY\to Z=\Pj^{n-1}$. The exceptional divisor over $q$ is a section. Let $v$ be a normalized divisorial valuation of $K(Y)$ whose restriction to the subfield $K(Z)$ is nontrivial. As in Subsection~\ref{subsec:valuation-basics}, write $v|_{K(Z)}=e\zeta$, where $e\in\mathbb Z_{>0}$ and $\zeta$ is a normalized discrete valuation of $K(Z)$. Proposition~\ref{prop:ramification} shows that $\zeta$ is divisorial.

For a prime $T$ representing $\zeta$ on a smooth model $Z'\to Z$, let $R=\cO_{Z',\eta_T}$, with uniformizer $\varpi$ and residue field $F_\zeta=\kappa(\zeta)$. Such a discrete valuation ring (DVR) is called a \textit{geometric DVR}. Set $\widetilde Y=(\Bl_qY)\times_Z Z'$. After choosing a trivialization over $R$, we have the commutative diagram with Cartesian squares
\[
\begin{array}{ccccc}
\Pj^1_R & \longrightarrow & \widetilde Y & \longrightarrow & \Bl_qY \\
\big\downarrow && \big\downarrow && \big\downarrow \\
\Spec R & \longrightarrow & Z' & \longrightarrow & Z.
\end{array}
\]
Indeed, the pullback to $\Spec R$ of the rank-two vector bundle defining $\Bl_qY\to Z$ is free, since $R$ is local. Thus $(\Bl_qY)\times_Z\Spec R\simeq\Pj^1_R$. We choose this identification so that the exceptional section is the section at infinity, and write $x$ for a coordinate on its complement $\mathbb A^1_R$.

On $\Pj^1_R$, a prime divisor is \textit{horizontal} if it dominates $\Spec R$, and \textit{vertical} otherwise. A valuation is vertical if $v(\varpi)>0$. The integer $\ord_D(\varpi)$ is the multiplicity of a vertical component $D$ in the fiber. For a horizontal divisor $H=\sum_i b_iH_i$, we use
\[
\begin{gathered}
\deg H^+_{\Frac R}=\sum_{b_i>0}b_i\deg(H_i|_{\Frac R}),\\
 \mu_P=\sum_i b_i\operatorname{mult}_P(H_i^{\mathrm{str}}),\quad
 \mu_P^+=\sum_{b_i>0}b_i\operatorname{mult}_P(H_i^{\mathrm{str}})
\end{gathered}
\]
on successive blowups at closed points of the closed fiber. Here $P$ is a closed point on one of these models, and $H_i^{\mathrm{str}}$ denotes the strict transform of $H_i$ on that model.
\end{notation}

\subsection{Monomial valuations and log discrepancies}\label{subsec:monomial-discrepancies}

We recall the local formula used in the next computation. On $V=\mathbb A^m$ with coordinates $z_1,\ldots,z_m$, choose positive rational weights $a_1,\ldots,a_m$. The corresponding \textit{monomial valuation} is
\[
 v_a\!\left(\sum_\beta c_\beta z^\beta\right)
 =\min_{c_\beta\ne0}\sum_{i=1}^m a_i\beta_i,
 \qquad c_\beta\in\C,
\]
extended to rational functions by subtraction on quotients. It is a positive multiple of a divisorial valuation, and its log discrepancy is the sum of its weights:
\begin{equation}\label{eq:monomial-discrepancy}
 A(v_a,V,0)=\sum_{i=1}^m a_i.
\end{equation}
The formula also holds in regular local parameters after localization on a smooth variety; see \cite[Proposition~5.1]{JM12}. For an $\R$-divisor $D=\sum_i d_i(z_i=0)$, subtracting its valuation gives
\[
 A(v_a,V,D)=A(v_a,V,0)-v_a(D)
           =\sum_{i=1}^m(1-d_i)a_i.
\]
These formulas retain the chosen weights: replacing $v_a$ by $cv_a$, for $c>0$, multiplies the log discrepancy by $c$.

For example, weights $(1,1)$ at a smooth surface point give the valuation of the exceptional divisor of the ordinary blowup, with log discrepancy $2$. More generally, weights $(1/r,t)$ in two regular parameters $s,z$ give log discrepancy $1/r+t$. This is the formula applied to the surface over $R'$ below, before the correction for ramification.

We also record how a surface monomial valuation with coprime positive integer weights $(a,b)$ is represented by ordinary point blowups. If $v(s)=a$ and $v(z)=b$ with $b>a$, blow up the center and use the chart with coordinates $(s,z/s)$; the weights become $(a,b-a)$. If $a>b$, use $(s/z,z)$, with weights $(a-b,b)$. Thus the successive weights follow the subtraction form of the Euclidean algorithm:
\[
 (a,b)\longmapsto
 \begin{cases}
 (a,b-a),&b>a,\\
 (a-b,b),&a>b.
 \end{cases}
\]
The greatest common divisor is preserved, so the sequence reaches $(1,1)$. Blowing up that final point extracts a prime divisor $E$ representing $v=\ord_E$. For example, $(2,5)\to(2,3)\to(2,1)\to(1,1)$, followed by the final blowup. This is what we mean by the \textit{Euclidean sequence of point blowups}.

This construction also verifies the discrepancy formula in dimension two. For one step $\sigma\colon S_1\to S_0$ with exceptional divisor $F$, the identity $K_{S_1}=\sigma^*K_{S_0}+F$ gives
\[
 A(v,S_0,0)=A(v,S_1,0)+v(F).
\]
At each subtraction step, $v(F)=\min\{a,b\}$ is exactly the decrease in the sum of the two weights. At the final point the weights are $(1,1)$ and the extracted divisor has discrepancy $2$ on the preceding surface. Reading these identities backwards gives $A(v,S_0,0)=a+b$ for the initial weights.

\subsection{Valuations of roots and polynomial coefficients}

\begin{notation}[Discrepancies and valuations of roots]\label{defn:contacts}
Let $R$ be a geometric DVR with uniformizer $\varpi$ and valuation $\zeta(\varpi)=1$. Projection leads us to study a horizontal divisor $H$ on the surface $\Pj^1_R$. We introduce its closed-fiber threshold and then compute it from equations for $H$.

Write $H=\sum_i b_i\operatorname{Div}(G_i)$, where $b_i\le1$ and $G_i\in R[x]$ are monic of positive degree, with squarefree product over $\Frac R$. The coefficients $b_i$ may be negative. For $d_i=\deg G_i$, let
\[
 G_i^{\mathrm h}(u,v)=v^{d_i}G_i(u/v)
\]
be the homogenization in homogeneous coordinates $[u:v]$ on $\Pj^1_R$, with $x=u/v$. We define $\operatorname{Div}(G_i)$ as the zero divisor of $G_i^{\mathrm h}$, viewed as a section of $\cO_{\Pj^1_R}(d_i)$. Since $G_i$ is monic, $G_i^{\mathrm h}(1,0)=1$, so this divisor is disjoint from the section at infinity $v=0$.

Put $S=\Pj^1_R$ and let $F=(\varpi=0)$ be its closed fiber. We define the \textit{closed-fiber threshold}, also called the local threshold below, by
\begin{equation}\label{eq:closed-fiber-threshold}
 \tau_R:=\sup\{c\in\mathbb R:(S,H+cF)\text{ is sub-lc over the closed point}\}
 =\inf_{w\ \mathrm{vertical}}\frac{A(w,S,H)}{w(\varpi)}.
\end{equation}
Here $w$ runs over divisorial valuations centered in $F$, so $w(\varpi)=w(F)>0$. The equality follows from
\[
 A(w,S,H+cF)=A(w,S,H)-c\,w(F).
\]
Thus each valuation bounds the allowed coefficient by $c\le A(w,S,H)/w(\varpi)$. Dividing by $w(\varpi)$ accounts for the different multiplicities of the closed fiber. We allow $c\in\mathbb R$, so $\tau_R$ may be negative when $(S,H)$ is not sub-lc over the closed point. When it is sub-lc there, $\tau_R\ge0$ and is the usual relative log canonical threshold of $F$ with respect to $(S,H)$, computed over the closed point of $\Spec R$.

Extend $\zeta$ to a splitting field, retaining the scaling $\zeta(\varpi)=1$, and let $R'$ be the valuation ring of this extension. The roots are integral, and the squarefree condition makes them distinct. Label all the roots $\alpha_j$. If $\alpha_j$ is a root of $G_i$, assign the section $x=\alpha_j$ the coefficient $b_j:=b_i$, inherited from the original component $H_i$. Thus $i$ indexes the original horizontal components, whereas $j$ indexes their individual roots. Over $R'$, the horizontal pullback is the sum of these sections with their assigned coefficients. Set
\[
 c_{jh}=\zeta(\alpha_j-\alpha_h),\qquad c_{jj}=\infty.
\]
For $j\ne h$, $c_{jh}$ measures the order to which the two sections agree along the closed fiber. A larger value means higher tangency.

We explain the calculation leading to the discrepancy formula. Put $S'=\Pj^1_{R'}$, let $g\colon S'\to S$ be the base change, and write $H'=g^*H$. Choose a uniformizer $s$ of $R'$, and let $r=\ord_s(\varpi)$ be the ramification index. Thus $\varpi$ is a unit times $s^r$. The extended valuation, still denoted by $\zeta$, satisfies
\[
 \zeta=\frac1r\ord_s,\qquad \zeta(s)=\frac1r,\qquad \zeta(\varpi)=1.
\]

Fix a root $\alpha_j$, put $z=x-\alpha_j$, and choose a rational $t>0$. We take the monomial valuation $w'$ with weights $1/r$ and $t$ in the parameters $s,z$. Explicitly, for $f_\ell\in\Frac R'$,
\[
 w'\!\left(\sum_\ell f_\ell z^\ell\right)
 =\min_{f_\ell\ne0}\{\zeta(f_\ell)+\ell t\}.
\]
We extend this rule to quotients by subtraction. It gives $w'(s)=1/r$, $w'(z)=t$, and $w'(\varpi)=1$. Since the weights are positive and rational, $w'$ is a positive multiple of a divisorial valuation. We use log discrepancy with the same scaling, so $A(cw',S',H')=cA(w',S',H')$ for $c>0$.

For another root $\alpha_h$, the identity $x-\alpha_h=z+(\alpha_j-\alpha_h)$ and the defining minimum above give
\[
 w'(x-\alpha_h)=\min\{t,\zeta(\alpha_j-\alpha_h)\}
               =\min\{t,c_{jh}\}.
\]
Even when the two values coincide, there is no cancellation: these are distinct powers of $z$, namely $z^1$ and $z^0$. Thus the equality uses the monomial choice of $w'$, not just the valuation inequality. Since $H'$ is the sum of the sections $x=\alpha_h$ with coefficients $b_h$, it follows that
\[
 w'(H')=\sum_h b_h\min\{t,c_{jh}\}.
\]
The elements $s,z$ form a regular system of parameters at their common zero. The monomial discrepancy formula \eqref{eq:monomial-discrepancy} gives $A(w',S',0)=w'(s)+w'(z)=1/r+t$. Subtracting the contribution of $H'$ therefore gives
\[
 A(w',S',H')=\frac1r+t-\sum_h b_h\min\{t,c_{jh}\}.
\]

It remains to compare this with the discrepancy over $S$. Let $F'=(s=0)$ be the closed fiber of $S'$. The base change is \emph{\'etale} over the generic fiber, so its canonical correction is supported on $F'$. In characteristic zero its coefficient is $r-1$, the order of the base Jacobian, as verified in Proposition~\ref{prop:contact}. With compatible choices of canonical divisors, we therefore have
\[
 K_{S'}=g^*K_S+(r-1)F',\qquad
 g^*(K_S+H)=K_{S'}+H'-(r-1)F'.
\]
The divisor in the crepant pullback is therefore $H'-(r-1)F'$. Replacing $H'$ by this divisor increases the log discrepancy by
\[
 (r-1)w'(F')=\frac{r-1}{r}.
\]
Let $w$ be the normalized divisorial valuation proportional to the restriction of $w'$ to $\Frac(R)(x)$. Since $w'(\varpi)=1$, this restriction is $w/w(\varpi)$. The discrepancy formula \eqref{eq:finite-contact} for the finite field extension, proved in Proposition~\ref{prop:contact}, now yields
\[
\begin{aligned}
 \frac{A(w,S,H)}{w(\varpi)}
 &=A\bigl(w',S',H'-(r-1)F'\bigr)\\
 &=\frac1r+t-\sum_h b_h\min\{t,c_{jh}\}+\frac{r-1}{r}\\
 &=1+t-\sum_h b_h\min\{t,c_{jh}\}.
\end{aligned}
\]
In particular, the constant $1$ includes the ramification correction: it is $1/r+(r-1)/r$. The term $t$ comes from the fiber coordinate, and the sum is the contribution of the horizontal divisor. Minimizing this expression motivates
\begin{equation}\label{eq:root-contact-max}
 I=\max_{j,\,t\ge0}\left\{\sum_h b_h\min\{t,c_{jh}\}-t\right\}.
\end{equation}
The case $t=0$ corresponds to the divisorial valuation of the closed fiber $(\varpi=0)\subset\Pj^1_R$ and gives zero in the maximum. The final slope for each $j$ is $b_j-1\le0$, so $I$ is finite and nonnegative. Set $I=0$ when there are no roots. Proposition~\ref{prop:contact} justifies the calculation, proves that these valuations compute the infimum, and gives
\[
 \tau_R=\inf_{w\ \mathrm{vertical}}\frac{A(w,\Pj^1_R,H)}{w(\varpi)}=1-I.
\]

For example, take $H=\tfrac34\operatorname{Div}(x)+\tfrac34\operatorname{Div}(x-\varpi^m)$ for an integer $m\ge1$. Then $c_{12}=m$, and the expression defining $I$ is $t/2$ for $0\le t\le m$ and $3m/4-t/4$ for $t\ge m$. Hence $I=m/2$ and $\tau_R=1-m/2$: higher tangency lowers the closed-fiber threshold.

Returning to the projection $Y\dashrightarrow Z$, let $\Theta$ be an $\R$-divisor on $Y$ with coefficients at most one and $q\notin\Supp\Theta$. Fix a normalized divisorial valuation $\zeta$ of $K(Z)$ and choose its geometric DVR $R$ as in Notation~\ref{defn:directions}. Let $H_\zeta$ be the horizontal divisor on $\Pj^1_R$ induced by $\Theta$, and write $H_\zeta=\sum_i b_i\operatorname{Div}(G_i)$ with $G_i\in R[x]$ monic, using the coefficients of the corresponding components of $\Theta$.

To interpret the contact function, first define
\[
 a_\Theta(\zeta):=\inf_{v|_{K(Z)}=e\zeta,\ e>0}
       \frac{A(v,Y,\Theta)}e.
\]
Taking $\Theta=0$ in this definition gives $a_0(\zeta)$. Proposition~\ref{prop:contact} identifies this zero-boundary value with the log discrepancy on the base:
\[
 a_0(\zeta)=A(\zeta,Z,0).
\]
The smooth closed fiber contributes $1$ and the base Jacobian contributes $A(\zeta,Z,0)-1$; the valuation of the closed fiber attains their sum.

We define $I_\Theta(\zeta)$ to be exactly the local number $I$ in \eqref{eq:root-contact-max}, applied to the data $(R,H_\zeta)$. Explicitly, extend $\zeta$ to a splitting field with $\zeta(\varpi)=1$, label the roots of the $G_i$ by $\alpha_h$, and assign each root the coefficient $b_h$ of its component. Then
\[
 I_\Theta(\zeta)
 :=\max_{j,\,t\ge0}\left\{\sum_h b_h\min\{t,c_{jh}\}-t\right\},
 \qquad c_{jh}=\zeta(\alpha_j-\alpha_h),\quad c_{jj}=\infty.
\]
Set $I_\Theta(\zeta)=0$ if there are no roots. Proposition~\ref{prop:contact} proves independence of the extension to the splitting field and gives
\[
 a_\Theta(\zeta)=a_0(\zeta)-I_\Theta(\zeta).
\]
Thus $I_\Theta(\zeta)=a_0(\zeta)-a_\Theta(\zeta)\ge0$ measures the decrease of the infimum of the discrepancy quotients $A(v,Y,\Theta)/e$ from the zero-boundary case. The horizontal local pair $(\Pj^1_R,H_\zeta)$ has closed-fiber threshold $1-I_\Theta(\zeta)$. The crepant pullback of $(Y,\Theta)$ has threshold $a_\Theta(\zeta)=a_0(\zeta)-I_\Theta(\zeta)$; the shift by $a_0(\zeta)-1$ comes from the base Jacobian. When $(Y,\Theta)$ is sub-lc, this latter threshold is the relative log canonical threshold along the closed fiber $F_R=(\varpi=0)$: adding $cF_R$ changes each discrepancy to $A(v,Y,\Theta)-ce$, so the largest allowed coefficient is precisely the infimum defining $a_\Theta(\zeta)$. In this setting, $I_\Theta(\zeta)$ is the loss of that threshold, computed by the contacts between the roots.
\end{notation}

\begin{notation}[Gauss valuations and ideals]\label{defn:coefficient-ideals}
We first specify the base and its valuation. As in Notation~\ref{defn:directions}, $Y=\Pj^n$, $n\ge2$, and $Z=\Pj^{n-1}$ is the base of projection from a point $q\in Y$. Let $\zeta$ be a normalized divisorial valuation of $K(Z)$: it measures order of vanishing along a prime divisor on a birational model of $Z$. In the projection construction, $\zeta$ comes from a divisorial valuation $v$ of $K(Y)$ whose restriction to $K(Z)$ is nonzero, with $v|_{K(Z)}=e\zeta$ for a positive integer $e$.

Let $s\ge1$ be an integer, and let $T=(T_1,\ldots,T_s)$ be algebraically independent auxiliary variables over $K(Z)$. For a multi-index $\lambda=(\lambda_1,\ldots,\lambda_s)\in\mathbb Z_{\ge0}^s$, write $T^\lambda=T_1^{\lambda_1}\cdots T_s^{\lambda_s}$. A polynomial $P(T)=\sum_\lambda a_\lambda T^\lambda$ is a finite sum with coefficients $a_\lambda\in K(Z)$. The \textit{Gauss extension} $\zeta_G$ is the valuation on $K(Z)(T_1,\ldots,T_s)$ defined, for $P\ne0$, by
\[
 \zeta_G(P)=\min_{a_\lambda\ne0}\zeta(a_\lambda),
 \qquad \zeta_G(P/Q)=\zeta_G(P)-\zeta_G(Q).
\]
Here $Q$ is any nonzero polynomial in $K(Z)[T_1,\ldots,T_s]$. Thus the value of a polynomial is the smallest value of its nonzero coefficients; in particular, $\zeta_G(T_i)=0$.

We apply this rule to coefficients varying on $Z$. Choose homogeneous coordinates $y=[y_0:\cdots:y_{n-1}]$ on $Z$, and consider a nonzero polynomial
\[
 C(y,T)=\sum_\lambda f_\lambda(y)T^\lambda.
\]
Here every nonzero coefficient $f_\lambda(y)$ is homogeneous of the same degree $d_0\in\mathbb Z_{\ge0}$: every monomial in $f_\lambda$ has total $y$-degree $d_0$. The powers $T^\lambda$ can have different degrees. For example, on $Z=\Pj^1$, the polynomial
\[
 C(y,T)=y_0^2T+y_0y_1T^3
\]
has two coefficients, $y_0^2$ and $y_0y_1$, both of degree $d_0=2$ in $y$.

Each $f_\lambda$ is a global section of the line bundle $\cO_Z(d_0)$. On the affine chart $U_i=(y_i\ne0)$, dividing by $y_i^{d_0}$ turns it into a regular function. The \textit{ideal sheaf generated by the coefficients} is therefore
\[
 \frakj=\operatorname{im}\!\left(
 \bigoplus_\lambda\cO_Z(-d_0)\xrightarrow{(f_\lambda)}\cO_Z
 \right),\qquad
 \frakj|_{U_i}=\left(\frac{f_\lambda}{y_i^{d_0}}\right)_\lambda.
\]
In the example, on $U_1$ with affine coordinate $u=y_0/y_1$, this ideal is $(u^2,u)=(u)$. If the generic point of the center of $\zeta$ lies in $U_i$, its valuation is
\[
 \zeta(\frakj)=\min_{f_\lambda\ne0}\zeta\!\left(\frac{f_\lambda}{y_i^{d_0}}\right)
 =\zeta_G\!\left(\frac{C(y,T)}{y_i^{d_0}}\right).
\]
This explains how the Gauss extension records the valuation of the coefficient ideal without assigning numerical values to $T$. The twist $\frakj(d_0)=\frakj\otimes\cO_Z(d_0)$ is globally generated, since twisting the defining map gives a surjection $\bigoplus_\lambda\cO_Z\to\frakj(d_0)$. In later bounds on $d_0/d$, the numerator $d_0$ is this common degree, and $d$ is the positive integer denominator in the ideal expressions of Proposition~\ref{prop:ideals}.

For the monomial valuations used later, let $z_1,\ldots,z_m$ be affine variables and let $a=(a_1,\ldots,a_m)\in\R_{\ge0}^m$ be their weights. For $\beta=(\beta_1,\ldots,\beta_m)\in\mathbb Z_{\ge0}^m$, put $z^\beta=\prod_i z_i^{\beta_i}$ and $\langle a,\beta\rangle=\sum_i a_i\beta_i$. With coefficients $c_\beta\in\C$, define
\[
 \zeta_a\!\left(\sum_\beta c_\beta z^\beta\right)
 =\min_{c_\beta\ne0}\langle a,\beta\rangle,
\]
and extend to $\C(z_1,\ldots,z_m)$ by subtraction on quotients. Thus $\zeta_a(z_i)=a_i$. We retain these weights without renormalizing; $a=0$ gives the trivial valuation, which is zero on every nonzero rational function.

Finally, fix polynomials $G_i(x,z_1,\ldots,z_m)$ that are monic of positive degree in $x$, with squarefree product over $\C(z)$, and fix coefficients $b_i\le1$, as in Notation~\ref{defn:contacts}. For nonzero rational $a$, $I(a)$ is the root-difference expression $I$ in that notation computed using $\zeta_a$, with its given scaling. The polynomials $G_i$ and coefficients $b_i$ remain fixed as $a$ varies. For real $a$, $I(a)$ denotes the continuous extension supplied by Corollary~\ref{cor:lipschitz}; in particular, $I(0)=0$.
\end{notation}

\subsection{Comparison on birational models}

We now compare log discrepancies on different birational models. On a smooth model $W$, an $\R$-divisor $D$ with simple normal crossing support and coefficients at most $1-\eta$, for $0<\eta\le1$, defines a sub-$\eta$-lc pair. To see this directly, if $z_1,\dots,z_r$ are local equations of the components of $D$ with positive coefficient that contain the center of $v$, then
\begin{equation}\label{eq:snc}
 A(v,W,0)\ge\sum_{i=1}^r v(z_i),\qquad A(v,W,0)\ge1.
\end{equation}
For the first inequality, the SNC condition lets us complete $z_1,\ldots,z_r$ to local coordinates $z_1,\ldots,z_m$, where $m=\dim W$, near a general point of the center of $v$. Take a smooth model $p\colon W'\to W$ containing the prime divisor $E$ with $v=\ord_E$. Near the generic point of $E$, write $p^*z_i=u^{m_i}a_i$ for $1\le i\le r$, where $u$ is a local equation of $E$, $a_i$ is a regular unit, and $m_i=v(z_i)$. Then
\[
 p^*\!\left(\frac{dz_i}{z_i}\right)
 =m_i\frac{du}{u}+\frac{da_i}{a_i}.
\]
The form $\omega=(dz_1\wedge\cdots\wedge dz_m)/(z_1\cdots z_r)$ therefore pulls back with at most a simple pole along $E$: its wedge expansion contains at most one factor $du/u$, since $du\wedge du=0$, and all remaining factors are regular. By the definition of log discrepancy,
\[
 -1\le\ord_E(p^*\omega)
 =A(v,W,0)-1-\sum_{i=1}^r v(z_i),
\]
which proves the first inequality. Pulling back the regular form $dz_1\wedge\cdots\wedge dz_m$ also gives $A(v,W,0)-1\ge0$, proving the second. If the positive coefficients are $b_i\le1-\eta$, these inequalities give
\[
 A(v,W,D)\ge A(v,W,0)-(1-\eta)\sum_i v(z_i)
 \ge\eta A(v,W,0)\ge\eta;
\]
negative coefficients only increase the discrepancy. This also shows that sub-$\eta$-lc can be checked on one log resolution. We use projective log resolutions and principalizations as in \cite[Theorems~26--27]{Kol05}.

\begin{lem}[General members of a base point free linear system]\label{lem:free-members}
On a smooth projective variety, a general member of a base point free linear system is smooth, and its union with any fixed reduced SNC divisor $\Sigma$ has simple normal crossings. It avoids containing finitely many prescribed nonempty proper subvarieties. Finitely many members may be chosen successively with these properties; empty members are allowed.
\end{lem}
\begin{proof}
Let $\mathcal L$ be the line bundle defining the linear system, and let $V\subseteq H^0(\mathcal L)$ be its finite-dimensional space of sections, which generates $\mathcal L$ at every point. Let $S$ run through the irreducible components of all nonempty intersections of components of $\Sigma$, together with the ambient variety. Each $S$ is smooth, and evaluation $V\otimes\cO_S\to\mathcal L|_S$ is surjective. The incidence scheme $\{(s,x)\in V\times S:s(x)=0\}$ is its kernel vector bundle, hence smooth. Its generic fiber over the affine parameter space $V$ is regular, since its local rings are localizations of regular rings. The parameter function field has characteristic zero, so this fiber is smooth by \cite[Tag~0B8X]{StacksBertiniSlice}; smoothness spreads to a nonempty open subset of the parameter space by \cite[Tag~0C0C]{StacksBertiniSlice}. Intersect these opens for the finitely many choices of $S$. Sections vanishing identically on any fixed nonempty subvariety form a proper linear subspace, since $V$ generates $\mathcal L$. Excluding these subspaces also ensures that the member misses every zero-dimensional $S$. Its intersection with every positive-dimensional $S$ is a smooth Cartier divisor, which is the local parameter criterion for SNC. Repeat after adding each member to $\Sigma$.
\end{proof}

\begin{lem}[Negativity and comparison of log discrepancies]\label{lem:comparison}
In Setup~\ref{setup:source}, $F$ is effective and exceptional, and
\begin{equation}\label{eq:correction}
 A(v,X,B)=A(v,Y,\Delta)+v(F)\ge A(v,Y,\Delta).
\end{equation}
Let $\rho\colon S\to Y$ be a projective birational morphism with $S$ smooth. Let $B_S$ be the divisor of Notation~\ref{defn:source-reference}. Then
\[
 A(v,X,B)\ge A(v,S,B_S)
\]
for every $v$, without a morphism $X\to S$ being required.
\end{lem}
\begin{proof}
The divisor $F$ is $\R$-Cartier, has zero pushforward, and $-F$ is $f$-nef. Negativity gives $F\ge0$; see \cite[Lemma~4.16]{Fujino2011}. Pulling its defining equality to a model representing $v$ gives \eqref{eq:correction}. For the assertion concerning $S$, take a smooth common resolution $p\colon W\to X$, $g\colon W\to S$, with $K_W+B_W=p^*L$ and $B_S=g_*B_W$. The divisor $G=g^*(K_S+B_S)-p^*L$ is $g$-exceptional, and $-G$ is $g$-nef because $g$-contracted curves map to points of $Y$. Negativity gives $G\ge0$. After replacing $W$ by a higher smooth model on which $v$ is represented by a prime divisor, the coefficient of the pullback of $G$ at that divisor is $A(v,X,B)-A(v,S,B_S)$.
\end{proof}

\begin{lem}[Rational approximation of the boundary]\label{lem:rational}
For each rational $0<\tau\le\eps$, the boundary $B$ is a coefficientwise limit of $\Q$-boundaries on the same support and on the same $X$, with $\tau$-lc pair, $\Q$-Cartier $f$-nef log canonical divisor, and degree of the pushforward at most $k$. Moreover $B$ belongs to a rational polytope all of whose points have these properties, with $\R$-Cartier log canonical divisors at real points.
\end{lem}
\begin{proof}
Use the rational affine coefficient space generated by the finitely many components of $B$. The locus of coefficients for which $K_X+B'$ is $\R$-Cartier is rational affine. Within this locus, the boundaries $B'$ for which $(X,B')$ is lc and $K_X+B'$ is $f$-nef form a rational polytope; this holds when $X$ is normal without $\Q$-factoriality by \cite[Theorem~5.2.3]{FujinoCC24FullPolytopeProof}, equivalently \cite[Theorem~10.8(3)]{FujinoQL22RationalNefPolytope}. For the projective algebraic morphism, the analytic formulation applies with the compact set $Y(\C)$: the numerical space is finite dimensional, and algebraic and analytic curve nefness agree.

On a fixed log resolution $p\colon W\to X$, write $K_W+B'_W=p^*(K_X+B')$. The coefficients of $B'_W$ are rational affine functions of the coefficients of $B'$ on the affine locus where $K_X+B'$ is $\R$-Cartier. Intersect the polytope with the inequalities that these coefficients are at most $1-\tau$, with nonnegativity, and with the rational linear inequality $\deg f_*B'\le k$. For each component of $\Supp B$, also impose a positive rational lower bound smaller than its coefficient in $B$. The resulting rational polytope contains $B$, all its boundaries have exactly the same support as $B$, and \eqref{eq:snc} gives the required lower bound at every valuation. Rational points are dense in it. A $\Q$-Weil divisor that is $\R$-Cartier is $\Q$-Cartier: after clearing its denominators, its class in each local divisor class group vanishes after tensoring with $\R$, hence also after tensoring with $\Q$, and is therefore torsion. Cartierness holds on a neighborhood of each point, and a finite open cover gives a common Cartier multiple. This proves the assertion.
\end{proof}

\begin{lem}[Projective degree estimate]\label{lem:degree}
For an effective $\R$-divisor $D$ on $\Pj^r$ and a divisorial valuation $v$,
\[
 v(D)\le(\deg D)A(v,\Pj^r,0).
\]
\end{lem}
\begin{proof}
On a standard affine chart $U=\mathbb A^r$ containing the generic center of $v$, a prime hypersurface has a nonzero polynomial equation $h$ of degree at most its projective degree. Let $x_1,\ldots,x_r$ be the coordinates on $U$, write $\partial_i=\partial/\partial x_i$, and put $A=A(v,U,0)$. We first justify the differentiation inequality
\[
 v(\partial_i h)\ge v(h)-A.
\]

Write $v=\ord_E$ and choose a smooth birational model $p\colon W\to U$ containing $E$. Near a general point of $E$, choose local coordinates $t,z_2,\ldots,z_r$ with $E=(t=0)$. The Jacobian determinant $J$ is defined by
\[
 p^*(dx_1\wedge\cdots\wedge dx_r)
 =J\,dt\wedge dz_2\wedge\cdots\wedge dz_r.
\]
Its order along $E$ is the coefficient of $K_W-p^*K_U$, hence $\ord_E(J)=A-1$ by the definition of log discrepancy.

Set $m=v(h)$ and write $p^*h=t^m u$, where $u$ is a regular unit near the generic point of $E$. Then
\[
 d(p^*h)=m t^{m-1}u\,dt+t^mdu,
\]
so every coefficient of this differential has order at least $m-1$ along $E$. All the differentials $p^*dx_j$ are regular. Consequently, replacing $dx_i$ by $dh$ in the top form and pulling back gives a form of order at least $m-1$. On the other hand,
\[
 dx_1\wedge\cdots\wedge dx_{i-1}\wedge dh
 \wedge dx_{i+1}\wedge\cdots\wedge dx_r
 =(\partial_i h)\,dx_1\wedge\cdots\wedge dx_r.
\]
Indeed, in $dh=\sum_j(\partial_jh)dx_j$, every term with $j\ne i$ repeats a differential and has zero wedge product. The order of the pulled-back right-hand side is $v(\partial_i h)+A-1$. Thus
\[
 v(\partial_i h)+A-1\ge m-1,
\]
which is the claimed inequality. If $\partial_i h=0$, it holds with the convention $v(0)=+\infty$.

Choose a monomial of highest total degree in $h$ and differentiate according to its exponents. The resulting iterated partial derivative is a nonzero constant, since the ground field is $\C$. Iterating the inequality at most $\deg h$ times gives $0\ge v(h)-(\deg h)A$, hence $v(h)\le(\deg h)A$. Summing over the prime components with the nonnegative real coefficients of $D$ proves the claim.
\end{proof}

\begin{prop}[An auxiliary sub-pair on the target]\label{prop:signed-target}
In Setup~\ref{setup:source}, assume $n\ge2$ and put $c=k+n$. For any finite collection $\mathcal V$ of divisorial valuations and any $u>0$, there is an effective $\R$-divisor $N$ on $Y$, with $\deg N=c$, no prime component in common with $\Delta$, and positive coefficients at most $u$, such that $(Y,\Theta)$ is sub-$\eps$-lc for $\Theta=\Delta-N$, and
\[
 A(v,Y,\Theta)\ge A(v,X,B)
\]
for all $v$, with equality at every $v\in\mathcal V$ and every $f$-exceptional prime divisor on $X$. If $B$ is a $\Q$-boundary, $N$ may be chosen as a $\Q$-divisor, namely a single prime hypersurface divided by an arbitrarily large sufficiently divisible integer.
\end{prop}
\begin{proof}
The cone theorem, in its form for normal pairs with $\R$-boundaries \cite[Theorems~1.1 and~18.10]{Fujino2011}, gives $L+2n f^*H$ nef. On an $L$-negative extremal ray its generating rational curve has $-L$-degree at most $2n$, and its $f^*H$-degree is a positive integer, since $L$ is relatively nef. Consequently
\[
 T_0=(k+n-1)f^*H-F=L+2nf^*H+(k-d)f^*H
\]
is nef. Thus $T=T_0+f^*H$ is nef and big.

Suppose first that $B$ is a $\Q$-boundary. Then $T$ is $\Q$-Cartier and
$T-L\sim_\Q(k+2n+1-d)f^*H$ is nef and big. The base point free theorem for klt pairs implies that a sufficiently divisible multiple of $T$ is globally generated; see \cite[Theorem~2.1]{Masamura23BPF}. Take $q$ sufficiently divisible so that $|qT|$ is the pullback of a very ample system on its normal semiample contraction. This contraction is birational because $T$ is big. On a smooth resolution of $X$, the pulled system has empty base locus. In its smooth isomorphism locus, two hyperplanes with independent differentials have a codimension-two intersection component; a third hyperplane avoids a chosen point of that component. Since $n\ge2$, \cite[Tag~0G4F]{StacksBertiniSlice} therefore makes the general member geometrically irreducible. Lemma~\ref{lem:free-members} makes it smooth and ensures that it contains none of the prescribed valuation centers, none of the strict transforms of the boundary components or of the $f$-exceptional prime divisors, and no exceptional prime divisor of the resolution. Its pushforward $D_X\in|qT|$ is thus a single prime divisor of multiplicity one. As an effective Cartier divisor on normal $X$, it is $S_1$ and generically reduced, hence reduced by \cite[Tags~0B3S and~0344]{StacksFoundations}. Thus $D_X$ is integral and contains none of the prescribed centers or divisor components. Put $M=D_X/q$ and $N=f_*M$. Then $N=D/q$ for a prime target hypersurface $D$. The equality $M\sim_\Q cf^*H-F$, after writing the linear equivalence as a sum of principal divisors and pushing to $Y$, gives the exact identity
\begin{equation}\label{eq:negative-realization}
 f^*N=F+M,\qquad \deg N=c.
\end{equation}
No prime component of $M$ is $f$-exceptional, by the choice of the general member. We spell out the role of the finite set $\mathcal V$ in this choice. For each $v\in\mathcal V$, let $Z_v$ be its center on $X$. The members of $|qT|$ containing $Z_v$ form a proper closed linear subspace, since $|qT|$ is base point free. As $\mathcal V$ is finite, the complement of the union of these subspaces is a nonempty open subset of $|qT|$. Thus the general member $D_X$ above can be chosen to contain none of the $Z_v$, simultaneously with the other conditions. Its pullback to the resolution then contains none of the centers of these valuations there. The local equation of $D_X$ is a unit at the generic point of each $Z_v$, and hence
\[
 v(M)=\frac{v(D_X)}q=0\qquad\text{for every }v\in\mathcal V.
\]
This simultaneous choice is where finiteness of $\mathcal V$ is used. Taking $q$ larger gives $1/q\le u$.

For an $\R$-boundary $B$, express it as a positive convex combination $\sum\lambda_i B_i$ of rational vertices in Lemma~\ref{lem:rational}, with any fixed positive rational lower bound for the log discrepancies. The associated corrections satisfy $F=\sum\lambda_iF_i$. Each $T_i=cf^*H-F_i$ is nef, big and semiample by the rational argument. Choose general members in sufficiently large multiples so that they contain none of the prescribed centers or divisor components and their images in $Y$ are pairwise distinct. Their weighted sum gives an effective $\R$-divisor $M\sim_\R cf^*H-F$ with arbitrarily small coefficients and $v(M)=0$ for every $v\in\mathcal V$, since each chosen member has value zero at all these valuations. Identity \eqref{eq:negative-realization} follows by pushing its $\R$-linear equivalence, exactly as above. Finally
\[
 A(v,Y,\Delta-N)=A(v,X,B)+v(M)\ge A(v,X,B).
\]
This proves the lower bound $\eps$ for log discrepancies, although the rational approximation used a smaller lower bound. For every $v\in\mathcal V$, the equality $v(M)=0$ makes the displayed discrepancy inequality an equality. The same holds for the valuation of every $f$-exceptional prime divisor on $X$, because $M$ has no such prime component.
\end{proof}

\begin{prop}[A relatively nef model with prescribed divisors]\label{prop:effective}
Let $S$ be smooth projective of dimension at least two, let $C=P-N$, where $P,N\ge0$ have no prime component in common, and assume $(S,C)$ is sub-$\eta$-lc, where $0<\eta\le1$. Any finite collection of divisorial valuations $E_i$ with centers of codimension at least two on $S$ and with $A(E_i,S,C)\le1$ is represented on a projective birational morphism $g\colon V\to S$, where $V$ is $\Q$-factorial and carries an effective $\eta$-lc boundary $C_V$ with $\R$-Cartier log canonical divisor $K_V+C_V$ nef over $S$, $g_*C_V=P$, and
$A(E_i,V,C_V)=A(E_i,S,C)$. Moreover
\begin{equation}\label{eq:effective-comparison}
 K_V+C_V=g^*(K_S+C)+G_V,\qquad G_V\ge0.
\end{equation}
\end{prop}
\begin{proof}
On a log resolution representing the selected valuations, write $K_W+C_W=p^*(K_S+C)$ and $C_W=\Gamma-G$ for its positive and negative parts. The effective SNC boundary $\Gamma$ is $\eta$-lc by \eqref{eq:snc}. We have $K_W+\Gamma=p^*(K_S+C)+G$ with $G\ge0$. Run a terminating relative log minimal model program for the effective klt pair $(W,\Gamma)$ over $S$, with scaling of a suitable effective relatively ample divisor chosen so that the initial scaled pair is klt and its log canonical divisor relatively nef. Since $W\to S$ is birational, termination follows from \cite[Theorem~1.2, Corollary~1.4.2]{BCHM10}; the corresponding existence statement is also \cite[Theorem~10.3]{Fujino2011}. Relative bigness is automatic, and the displayed effective representative gives relative pseudo-effectivity.

A selected $E_i$ has coefficient zero in $G$, since its coefficient in $C_W$ is $1-A(E_i,S,C)\ge0$. At every intermediate model the same divisor identity holds for the transforms of $\Gamma$ and $G$: take strict transforms under flips and pushforwards under divisorial contractions. A flip does not contract any prime divisor. If a divisorial contraction removed $E_i$, a contracted curve through its general point would not be contained in $\Supp G$ and would have nonnegative $G$-degree, since $G\ge0$. The log canonical divisor has this same degree, contradicting negativity of the contracted extremal ray. Thus every selected prime survives. Discrepancies increase along the program, while the coefficient of each surviving prime is unchanged. The final boundary is effective and $\eta$-lc, its pushforward is $P$, and its identity is \eqref{eq:effective-comparison}. The effective divisor $G_V$ may contain the strict transform of $N$; it need not be exceptional.
\end{proof}

\section{Projection and uniform ramification}\label{sec:ramification}

Under projection to $\Pj^{n-1}$, the valuation of a prime divisor on $X$ contracted by $f$ restricts to a positive integral multiple of a divisorial valuation. We bound this multiple uniformly.

\begin{prop}[A uniform bound under projection]\label{prop:ramification}
In Setup~\ref{setup:source}, assume $n\ge2$. Project from a point $q$ to $Z=\Pj^{n-1}$. Let $E_q$ be the exceptional divisor of the ordinary blowup at $q$, and set
\[
 S=k+\max\{0,1-A(E_q,X,B)\},\quad M=\lceil2S/\eps\rceil,\quad \delta=\eps/(2M).
\]
Let $v$ be a normalized divisorial valuation of $K(Y)=K(X)$ whose restriction to the subfield $K(Z)$ is nontrivial. Then $v|_{K(Z)}=e\zeta$ for a positive integer $e$ and a normalized divisorial valuation $\zeta$ of $K(Z)$. These valuations satisfy
\begin{equation}\label{eq:ramification}
 A(v,X,B)\ge\delta e.
\end{equation}
Here $S\le k+1-\eps$. If $q$ belongs to the isomorphism locus of $f$ and avoids $\Supp\Delta$, the valuation of every prime divisor on $X$ contracted by $f$ has nontrivial restriction, and one can use $S=k$. In that case the integer $e$ is at most
\[
 R(k,\eps)=\left\lfloor\frac{2\lceil2k/\eps\rceil}{\eps}\right\rfloor.
\]
\end{prop}
\begin{proof}
A nonzero subgroup of the value group $\mathbb Z$ is $e\mathbb Z$. Normalize the restriction accordingly. Since $K(Y)=K(Z)(x)$, the residue field of $v$ has transcendence degree at most one over the residue field of $\zeta$: algebraically independent residues of two value-zero elements would lift to algebraically independent elements over $K(Z)$, as follows by reducing a relation after dividing by a coefficient of minimum value. Thus the residue field of $\zeta$ has transcendence degree at least $n-2$. On the other hand, the Abhyankar inequality for the nontrivial discrete valuation $\zeta$ gives $\trdeg_\C\kappa(\zeta)\le\trdeg_\C K(Z)-1=n-2$; see \cite[Section~3.2]{JM12}. Hence $\zeta$ is divisorial. For example, lift a residue transcendence basis and adjoin a positive-value element; normalization in the resulting finite field extension represents $\zeta$ as a prime divisor.

Represent $\zeta=\ord_T$ by a prime divisor $T$ on a smooth projective model $\rho\colon Z'\to Z$. Localize at its generic point, obtaining the geometric DVR $R=\cO_{Z',\eta_T}$ with uniformizer $\varpi$ and residue field $F_\zeta$. Set $\widetilde Y=(\Bl_qY)\times_Z Z'$. The projection and its localization fit into the following commutative diagram, whose two left squares are Cartesian:
\begin{equation}\label{diag:ramification-localization}
\begin{tikzcd}[column sep=large,row sep=large]
 \Pj^1_R \arrow[r] \arrow[d] &
 \widetilde Y \arrow[r] \arrow[d] &
 \Bl_qY \arrow[r,"b"] \arrow[d,"p"'] &
 Y \arrow[dl,dashed,"\operatorname{pr}_q"] \\
 \Spec R \arrow[r] & Z' \arrow[r,"\rho"'] & Z &
\end{tikzcd}
\end{equation}
Here $b$ is the blowup at $q$, and $p$ is the resulting $\Pj^1$-bundle. Thus passing to $\Spec R$ replaces the $n$-dimensional projection space by the regular surface $\Pj^1_R$. The divisor $T$ on the base becomes the closed point of $\Spec R$, and valuations restricting to a positive multiple of $\zeta$ become vertical valuations over this surface. The exceptional divisor $E_q$ gives its section at infinity.

The horizontal part of the divisor $B_{\widetilde Y}$ of Notation~\ref{defn:source-reference} has positive part of degree at most $S$ on the generic fiber: a degree-$d$ hypersurface on $Y$ has degree at most $d$ in the fiber coordinate, and the only additional horizontal exceptional divisor is the section at infinity, with coefficient $1-A(E_q,X,B)$. All negative horizontal coefficients are retained.

Starting with $\mathcal S_0=\Pj^1_R$, consider successive blowups $\sigma_i\colon\mathcal S_{i+1}=\Bl_{P_i}\mathcal S_i\to\mathcal S_i$ at closed points $P_i$ of the closed fiber. All these regular surface models lie over the same DVR:
\[
\begin{tikzcd}[column sep=large,row sep=large]
 \mathcal S_{i+1} \arrow[r,"\sigma_i"] \arrow[d] & \mathcal S_i \arrow[d] \\
 \Spec R \arrow[r,equal] & \Spec R.
\end{tikzcd}
\]
Such a finite sequence spreads to smooth projective birational models over $Y$: spread the point ideals, take their closures and blow up, then resolve outside the smooth generic surface. For each vertical component $D$, write $m_D=\ord_D(\varpi)$ and $a_D=A(D,X,B)\ge\eps$. In particular, the closed fiber on $\mathcal S_i$ is the divisor $\sum_Dm_DD=\operatorname{div}_{\mathcal S_i}(\varpi)$, while its reduced support remains SNC. We will prove $a_D\ge\delta m_D$ for every such component; for a component representing $v$, its fiber multiplicity is exactly the restriction index $e$.

On the current surface model over $\Spec R$, let $P$ be the closed point of the closed fiber at which the next blowup is performed. Its residue field $\kappa(P)$ is a finite extension of $F_\zeta$. Exactly one or two irreducible components of the reduced closed fiber pass through $P$; denote them by $D$ or by $D,D'$, respectively. Horizontal components may also pass through $P$, and their multiplicities with the given coefficients enter $\mu_P$ below. Let $E$ denote the exceptional divisor extracted by this blowup, so $a_E=A(E,X,B)$ and $m_E=\ord_E(\varpi)$. Lemma~\ref{lem:comparison} and the canonical divisor formula for a point blowup give
\begin{equation}\label{eq:birth}
 a_E\ge
 \begin{cases}a_D+1-\mu_P,&\text{only }D\text{ passes through }P,\\
 a_D+a_{D'}-\mu_P,&D,D'\text{ pass through }P,\end{cases}
 \qquad
 m_E=\begin{cases}m_D,\\m_D+m_{D'},\end{cases}
\end{equation}
where $\mu_P$ is the multiplicity at $P$ of the strict transform of the horizontal $\R$-divisor, with its given coefficients. Recall that, writing this horizontal divisor as $H=\sum_i b_iH_i$, we have
\[
 \mu_P=\sum_i b_i\operatorname{mult}_P(H_i^{\mathrm{str}}),\qquad
 \mu_P^+=\sum_{b_i>0}b_i\operatorname{mult}_P(H_i^{\mathrm{str}}).
\]
Here $H_i^{\mathrm{str}}$ denotes the strict transform on the current surface model. Thus $\mu_P^+$ is the multiplicity at $P$ of the strict transform of the positive horizontal part $H^+=\sum_{b_i>0}b_iH_i$, and $\mu_P\le\mu_P^+$.

The strict transform of each horizontal prime divisor is finite and flat over $R$, so its special-fiber length equals its generic degree, including residue-field degrees. Consequently
$\sum_Dm_D\deg_{F_\zeta}(H^+|D)\le S$, where $H^+$ is understood to be its strict transform on the current surface. For each fiber branch $D$ through $P$, let $i_P(H^+,D)$ denote the local intersection length, extended linearly in the coefficients of $H^+$. Since $D$ is smooth at $P$, we have $i_P(H^+,D)\ge\mu_P^+$. The contribution of $P$ to the weighted intersection sum $\sum_Dm_D\deg_{F_\zeta}(H^+|D)$ is therefore
\[
 [\kappa(P):F_\zeta]\sum_{D\ni P}m_D\,i_P(H^+,D)
 \ge [\kappa(P):F_\zeta]\Bigl(\sum_{D\ni P}m_D\Bigr)\mu_P^+
 = [\kappa(P):F_\zeta]m_E\mu_P^+,
\]
where the last equality uses \eqref{eq:birth}. All other local contributions are nonnegative, so this contribution is at most $S$. In particular
\[
 \mu_P\le\mu_P^+\le S/m_E.
\]
This estimate is independent of the degree, on the generic fiber, of the negative part of the horizontal divisor.

We show inductively that $a_D\ge\delta m_D$, with the stronger inequality $a_D\ge\delta m_D+\eps/2$ whenever $m_D>M$. For $m_D\le M$, use $a_D\ge\eps$ and $\delta M=\eps/2$. If a blowup produces $m_E>M$, the term $\mu_P$ is at most $\eps/2$. If only one fiber component passes through $P$, its multiplicity is unchanged, so the stronger inequality persists after adding $1-\mu_P$. If two fiber components pass through $P$ and both multiplicities are at most $M$, their total discrepancy is at least $2\eps$ and the new multiplicity is at most $2M$, which gives the stronger inequality after subtraction. If exactly one of their multiplicities exceeds $M$, its stronger inequality and the other component's discrepancy bound $\eps$ absorb the term $\mu_P$ and the at most $\eps/2$ required by the smaller multiplicity. If both multiplicities exceed $M$, the sum of the two stronger inequalities still gives an additional $\eps/2$ after subtracting $\mu_P$. This proves the induction.

Every vertical divisorial valuation of $K(Z)(x)$ is represented by such a finite point-blowup sequence. One way to check finiteness is to choose a value-zero rational function whose residue is transcendental over $F_\zeta$, and principalize the base ideal of the pencil defined by a numerator and denominator of this function. On the resulting regular surface the valuation center maps to the generic point of the closed $\Pj^1_{F_\zeta}$, hence is a fiber divisor. Its DVR is the given valuation ring. Principalization is available by \cite[Theorem~26]{Kol05}; after localization the nontrivial smooth centers are points. Its fiber multiplicity is $v(\varpi)=e$, proving \eqref{eq:ramification}.

For the special choice of $q$, one has $A(E_q,X,B)=n$, so the extra horizontal coefficient is nonpositive. The image of a prime divisor $E$ on $X$ contracted by $f$ has dimension at most $n-2$ and avoids $q$. Its projected image is a proper subset of $Z$, so the restriction is nontrivial. Since $A(E,X,B)\le1$ by effectivity, \eqref{eq:ramification} gives the stated upper bound for $e$.
\end{proof}

\section{Discrepancies under projection}\label{sec:contacts}

We compute $a_\Theta$ and express $I_\Theta$ in terms of coherent ideals on $Z$. Using the ideals generated by the coefficients in auxiliary parameters makes the resulting identity valid simultaneously for all divisorial valuations.

\begin{prop}[A formula for the infimum of discrepancies]\label{prop:contact}
With the local surface, horizontal divisor $H$, and function $I$ of Notation~\ref{defn:contacts}, one has
\begin{equation}\label{eq:local-contact}
 \inf_{w\ \mathrm{vertical}}\frac{A(w,\Pj^1_R,H)}{w(\varpi)}=1-I.
\end{equation}
The infimum is attained and is independent of the chosen extension of $\zeta$ to the splitting field. Adding a nonpositive multiple of the section at infinity does not change it.

For an $\R$-divisor $\Theta$ on $Y=\Pj^n$ with coefficients at most one and $q\notin\Supp\Theta$, the functions $a_\Theta$ and $I_\Theta$ of Notation~\ref{defn:contacts} satisfy
\begin{equation}\label{eq:direction-contact}
 a_\Theta(\zeta)=a_0(\zeta)-I_\Theta(\zeta).
\end{equation}
\end{prop}
\begin{proof}
Monicity makes every root integral; all $c_{jh}$ with $j\ne h$ are finite nonnegative rational numbers. For each root $\alpha_j$, the expression defining $I$ is continuous and piecewise affine in $t$, with break points among the values $c_{jh}$ and final slope $b_j-1\le0$. Its maximum is finite, nonnegative and attained at zero or at one of these rational values.

We first compute the relative canonical divisor under the finite base change. For an extension of geometric DVRs $A\subset B$ with a finite extension of fraction fields, choose uniformizers $t,u$ and write $t=b u^a$ with $b\in B^\times$. Lift a separating transcendence basis of the residue field of $A$ over the ground field. Its rational function field $F$ lies in $A$: every nonzero denominator has nonzero residue. Their fraction fields have transcendence degree one over $F$, and their residue fields are finite separable over $F$. As geometric regular local rings in characteristic zero, they are therefore smooth curve local rings over $F$. Thus $dt$ and $du$ generate their respective cotangent modules by \cite[Tag~0C1E]{StacksFoundations}. Writing $db=c\,du$, with $c\in B$, gives
\[
 dt=u^{a-1}(ab+uc)\,du.
\]
The parenthesized factor is a unit in residue characteristic zero. The differential determinant therefore has order $a-1$; by \cite[Tag~0BWJ]{StacksFoundations} this is the canonical correction. Wedging with the differentials of the chosen lifts of the residue-field transcendence basis gives the same coefficient for top forms over $\C$; the finite separable residue extension contributes no extra vanishing. On valuation models the same computation shows that a finite crepant pullback multiplies log discrepancy by the valuation ramification index.

Take a finite Galois splitting field, a DVR $R'$ above $R$, and its normalized valuation $\zeta'$, with $h=\zeta'(\varpi)$. For an extension $w'$ of $w$, write $w'|\Frac(R)(x)=r w$ and $w'|\Frac(R')=e'\zeta'$. If $e=w(\varpi)$, then $he'=re$. The calculation just given shows that the divisor in the crepant pullback is $H'-(h-1)F'$, where $F'$ is the closed fiber and $w'(F')=e'$. Its discrepancy equals $r$ times the downstairs discrepancy, giving
\begin{equation}\label{eq:finite-contact}
 rA(w,\Pj^1_R,H)=A(w',\Pj^1_{R'},H')+(h-1)e'.
\end{equation}
The pullbacks of the horizontal prime divisors are sums of the sections defined by their roots, with the original coefficients. All divisor orders in \eqref{eq:finite-contact} enter linearly, so the equality remains valid for arbitrary real coefficients.

If there are no roots, then $H=0$ and $I=0$. Applying \eqref{eq:snc} to the smooth closed fiber $F=(\varpi=0)$ gives $A(w,\Pj^1_R,0)\ge w(F)=e$. The valuation $\ord_F$ has discrepancy one and fiber multiplicity one, so it attains the infimum $1$. This value is independent of any splitting-field extension and is unchanged by a nonpositive coefficient at infinity, since $\ord_F$ has order zero there. Thus all the local assertions hold in this case. For the remaining local calculation, assume that there is at least one root. In the affine chart $\mathbb A^1_{R'}=\Spec R'[x]$, complementary to the section at infinity, $x$ is the coordinate along the projection fiber. Choose a root $\alpha_j$ maximizing $u=w'(x-\alpha_j)/(re)$. Recall the valuation inequality
\[
 w'(f+g)\ge\min\{w'(f),w'(g)\},
\]
with equality whenever $w'(f)\ne w'(g)$. The extension of the base valuation used to define $c_{j\ell}$ is $\zeta'/h$. Since $w'|\Frac(R')=e'\zeta'$ and $he'=re$, we have
\[
 \frac{w'(\alpha_j-\alpha_\ell)}{re}
 =\frac{\zeta'(\alpha_j-\alpha_\ell)}h=c_{j\ell}.
\]
Apply the valuation inequality to $x-\alpha_\ell=(x-\alpha_j)+(\alpha_j-\alpha_\ell)$. It gives $w'(x-\alpha_\ell)/(re)\ge\min\{u,c_{j\ell}\}$. If $u\ne c_{j\ell}$, equality follows from the unequal-value case above. If $u=c_{j\ell}$, the inequality gives a lower bound $u$, while the maximal choice of $j$ gives the upper bound $w'(x-\alpha_\ell)/(re)\le u$. Thus in both cases
\[
 \frac{w'(x-\alpha_\ell)}{re}=\min\{u,c_{j\ell}\}.
\]
This also includes $\ell=j$, since $c_{jj}=\infty$.

To apply the parameter inequality, let $s$ be a uniformizer of $R'$ and put $z=x-\alpha_j$. Then $w'(s)=e'$ and $w'(z)=re u$. If $u>0$, the center of $w'$ is the closed point $s=z=0$, where $s,z$ form a regular system of parameters. Applying \eqref{eq:snc} to these two parameters gives
\[
 A(w',\Pj^1_{R'},0)\ge w'(s)+w'(z)=e'+re u.
\]
If $u=0$, the same bound follows by applying \eqref{eq:snc} just to the smooth closed fiber $F'=(s=0)$.

The split horizontal divisor is $H'=\sum_\ell b_\ell(x=\alpha_\ell)$. The root-order identity above therefore gives
\[
 w'(H')=\sum_\ell b_\ell w'(x-\alpha_\ell)
       =re\sum_\ell b_\ell\min\{u,c_{j\ell}\}.
\]
Subtract this boundary order from the discrepancy without boundary, and use \eqref{eq:finite-contact} to return to the original surface:
\[
\begin{aligned}
 \frac{A(w,\Pj^1_R,H)}e
 &=\frac{A(w',\Pj^1_{R'},0)-w'(H')+(h-1)e'}{re}\\
 &\ge\frac{he'}{re}+u-\sum_\ell b_\ell\min\{u,c_{j\ell}\}\\
 &=1-\left(\sum_\ell b_\ell\min\{u,c_{j\ell}\}-u\right)\\
 &\ge1-I.
\end{aligned}
\]
Here $he'=re$ explains the constant $1$. The last inequality follows from the definition of $I$: the expression in parentheses is one of the values included in its maximum, for the chosen root $j$ and $t=u\ge0$.

Finally, if the center of $w$ lies on the section at infinity, the monic equations defining $H$ imply that $H$ is disjoint from that section, so $w(H)=0$. The smooth closed fiber $F=(\varpi=0)$ still gives $A(w,\Pj^1_R,0)\ge w(F)=e$. Hence
\[
 \frac{A(w,\Pj^1_R,H)}e
 =\frac{A(w,\Pj^1_R,0)}e\ge1\ge1-I,
\]
where the final inequality uses $I\ge0$. This proves the lower bound for every vertical valuation.

To prove that the lower bound is attained, take a maximizing pair $(j,t)$ in the definition of $I$. If $t>0$, choose coprime positive integers $e',u'$ with $u'/e'=ht$. Let $w'$ be the monomial valuation of weights $(e',u')$ in a uniformizer $s$ of $R'$ and $z=x-\alpha_j$. The Euclidean sequence of point blowups described in Subsection~\ref{subsec:monomial-discrepancies} represents $w'$ by a prime divisor and gives
\[
 A(w',\Pj^1_{R'},0)=e'+u'.
\]
Its root orders are $w'(x-\alpha_\ell)=\min\{u',e'\zeta'(\alpha_j-\alpha_\ell)\}$, by the defining minimum rule for a monomial valuation. Restricting $w'$ to the original field gives a positive integral multiple $r w$ of a normalized divisorial valuation. With $e=w(\varpi)$, we again have $re=he'$. Since $u'/e'=ht$, formula \eqref{eq:finite-contact} now gives
\[
 \frac{A(w,\Pj^1_R,H)}e
 =1+t-\sum_\ell b_\ell\min\{t,c_{j\ell}\}=1-I.
\]
If $t=0$, then $I=0$ and the valuation of the closed fiber of $\Pj^1_R$ attains the same bound. Thus the infimum in \eqref{eq:local-contact} is attained.

Galois conjugacy permutes the roots of each polynomial and preserves their coefficients $b_j$, so the answer is independent of the chosen extension of $\zeta$ to the splitting field. The valuation attaining the infimum has center in the affine chart $\mathbb A^1_R$, so it has order zero along the section at infinity, whereas a negative coefficient of that section can only raise other discrepancies.

For the global statement, represent $\zeta=\ord_T$ on a smooth projective model $\rho\colon Z'\to Z$ and set $R=\cO_{Z',\eta_T}$. Pull back the bundle $\Bl_qY\to Z$ to $Z'$ and localize at $T$, as in \eqref{diag:ramification-localization}, obtaining $\Pj^1_R$. Denote its closed fiber by $F_R=(\varpi=0)$, its section at infinity by $\Sigma_\infty$, and the horizontal divisor induced by $\Theta$ by $H_\zeta$. Since $q\notin\Supp\Theta$, the component equations are monic in the finite fiber coordinate. Thus $H_\zeta$ is disjoint from $\Sigma_\infty$, and this base change introduces no additional vertical component in the pullback of $\Theta$.

There are two canonical divisor corrections. For the blowup morphism $b\colon\Bl_qY\to Y$, we have $K_{\Bl_qY}=b^*K_Y+(n-1)E_q$. Since the pullback of $\Theta$ has coefficient zero along $E_q$, the crepant boundary has coefficient $-(n-1)$ along the corresponding section $\Sigma_\infty$. For the change of base, the coefficient of $T$ in $K_{Z'}-\rho^*K_Z$ is $A(\zeta,Z,0)-1$. The bundle is locally a product, so the same Jacobian order occurs along $F_R$.

First take $\Theta=0$. The parameter inequality \eqref{eq:snc} gives $A(v,\Pj^1_R,0)\ge e$, and the contribution $(n-1)v(\Sigma_\infty)$ to the discrepancy is nonnegative. Adding the base Jacobian contribution therefore gives
\[
 \frac{A(v,Y,0)}e\ge 1+\bigl(A(\zeta,Z,0)-1\bigr)=A(\zeta,Z,0).
\]
The valuation $v=\ord_{F_R}$, represented by the divisor over $T$ in the pulled-back bundle, attains equality: it has $e=1$, local log discrepancy $1$, and order zero along $\Sigma_\infty$. Taking the infimum proves $a_0(\zeta)=A(\zeta,Z,0)$.

For general $\Theta$, we may consequently write the crepant boundary on the local surface as
\[
 H_\zeta-(n-1)\Sigma_\infty-(a_0(\zeta)-1)F_R.
\]
For any normalized divisorial valuation $v$ with $v|_{K(Z)}=e\zeta$, we have $v(F_R)=v(\varpi)=e$. Subtracting the boundary orders therefore gives
\[
\begin{aligned}
 A(v,Y,\Theta)
 &=A(v,\Pj^1_R,H_\zeta)\\
 &\quad +(a_0(\zeta)-1)e+(n-1)v(\Sigma_\infty).
\end{aligned}
\]
This explains the contribution $(a_0(\zeta)-1)e$ of the base Jacobian.

The local calculation gives infimum $1-I_\Theta(\zeta)$ for $A(v,\Pj^1_R,H_\zeta)/e$. Its minimizing valuation has center in the affine chart $\mathbb A^1_R$, so its order along $\Sigma_\infty$ is zero. This valuation is also allowed in the global infimum: restrict the construction from the finite splitting field to $K(Y)$ and normalize its value group. Its residue field has transcendence degree one over $\kappa(\zeta)$; since $\trdeg_\C\kappa(\zeta)=n-2$, its residue field has transcendence degree $n-1$ over $\C$. Hence it is divisorial over $Y$ and can be represented by a prime divisor on a smooth projective birational model of $Y$.

The term $(n-1)v(\Sigma_\infty)$ is nonnegative for every valuation under consideration and vanishes for this minimizing valuation. Dividing the discrepancy comparison by $e$ and taking the infimum thus gives
\[
 a_\Theta(\zeta)
 =\bigl(a_0(\zeta)-1\bigr)+\bigl(1-I_\Theta(\zeta)\bigr)
 =a_0(\zeta)-I_\Theta(\zeta),
\]
which is \eqref{eq:direction-contact}.
\end{proof}

\subsection*{Local thresholds on a log resolution}

The local threshold can be computed by a log resolution. We first separate this calculation, and the finite data needed for the degree estimate, from the construction of ideals on the base.

\begin{lem}[Finite data for the local threshold]\label{lem:local-threshold-data}
Use the local setup $(R,H)$ of Notation~\ref{defn:contacts}, with rational coefficients at most one. Label its roots by $\alpha_h$, with coefficients $b_h$ and contacts $c_{jh}$. Choose a positive integer $m$ clearing these coefficients, and count each positive root $mb_h$ times and each negative root $m|b_h|$ times. Put
\[
 P:=\sum_{b_h>0}mb_h,\qquad N:=\sum_{b_h<0}m|b_h|.
\]
For a positive center $j$, let $A_j(q)$ be the sum of the $q$ smallest positive-occurrence contacts and $B_j(\ell)$ the sum of the $\ell$ smallest negative-occurrence contacts; empty sums are zero. If $P\le m$, then $I=0$. If $P>m$, then
\begin{equation}\label{eq:order-statistics}
 mI=\max\{0,A_j(q_\ell)-B_j(\ell)\}_{j,\ell},\qquad
 q_\ell:=P-m-N+\ell,
\end{equation}
where $j$ ranges over positive roots and $\max\{0,m+N-P\}\le\ell\le N$. If $I>0$, some $\ell_0$ satisfies $A_j(q_{\ell_0})-B_j(\ell_0)=mI$ simultaneously for more than $m$ positive occurrences.
\end{lem}
\begin{proof}
Put $S=\Pj^1_R$ and let $F=(\varpi=0)$ be its closed fiber. Take a log resolution $\sigma\colon W\to S$ of $H+F$, and write
\[
 K_W+H_W=\sigma^*(K_S+H),\qquad
 b_E:=\operatorname{coeff}_E H_W,\qquad
 n_E:=\ord_E(\sigma^*F).
\]
The horizontal coefficients are at most one. By the SNC criterion \eqref{eq:snc}, adding $cF$ is sub-lc over the closed point precisely when $b_E+cn_E\le1$ for every vertical component on $W$. The definition \eqref{eq:closed-fiber-threshold} therefore gives
\begin{equation}\label{eq:local-threshold-resolution}
 \tau_R=\min_{n_E>0}\frac{1-b_E}{n_E}=1-I,
\end{equation}
where the last equality is \eqref{eq:local-contact}. The minimum includes the strict transform of $F$.

To read this resolution in terms of the horizontal sections, take a splitting DVR $R'$ with ramification index $r$, put $S'=\Pj^1_{R'}$, and denote the horizontal pullback by $H'$. Write $\tau_{R'}$ for the corresponding closed-fiber threshold of $(S',H')$. Finite crepant pullback, as in \eqref{eq:finite-contact}, gives
\[
 \tau_R=\frac{\tau_{R'}+r-1}{r}.
\]
Choose a uniformizer $s$ of $R'$. Then $\ord_s(\alpha_j-\alpha_h)=r c_{jh}$. Resolve $H'+F'$ by successive point blowups separating its sections. They are transverse to the closed fiber, and after each blowup their strict transforms meet the new vertical component away from fiber nodes. Hence every vertical component in this resolution has multiplicity one in the pullback of $F'$.

For a component created at depth $d$ along the section $\alpha_j$, a chart is $x=\alpha_j+s^d u$. Its discrepancy and root orders are
\[
 A(E,S',0)=1+d,\qquad
 \ord_E(x-\alpha_h)=\min\{d,r c_{jh}\}.
\]
Writing $\Phi_j(t):=\sum_h b_h\min\{t,c_{jh}\}-t$, we obtain
\[
 A(E,S',H')=1+d-\sum_h b_h\min\{d,r c_{jh}\}
 =1-r\Phi_j(d/r).
\]
Thus root contacts record the pullback orders on the resolution. The threshold itself is the minimum of the discrepancies of these vertical components, since their fiber multiplicities are one.

Suppose $I>0$. Then $\tau_{R'}=1-rI<1$. Choose a vertical component attaining this minimum at the smallest possible depth $d$, and put $t_0:=d/r$. If $D$ is its parent component, the discrepancy of $D$ is strictly larger. The sections through the blowup center are exactly
\[
 \mathcal C_{j_0}(t_0):=\{\alpha_h:c_{j_0h}\ge t_0\}
\]
for any section $\alpha_{j_0}$ through that center. The blowup formula gives
\[
 A(E,S',H')=A(D,S',H')+1-
 \sum_{\alpha_h\in\mathcal C_{j_0}(t_0)}b_h.
\]
Consequently this set has total coefficient greater than one. Each smaller set of sections still meeting at a point on $E$ has total coefficient at most one: otherwise the next blowup would create a component with smaller discrepancy. A singleton also satisfies this bound because its coefficient is at most one. Every root in the original set gives $\Phi_j(t_0)=I$, by the chart formula. In particular a positive center attains $I$. If $I=0$, the maximum with zero is already attained at $t=0$; hence positive centers suffice in \eqref{eq:order-statistics} in all cases.

We next replace the continuous parameter by finite selections of pullback orders. In the following sums $h$ indexes the original roots, with their coefficients $b_h$. The repeated negative contacts satisfy
\[
 m\sum_{h:\,b_h<0}|b_h|\min\{t,c_{jh}\}
 =\min_{0\le\ell\le N}\bigl\{B_j(\ell)+(N-\ell)t\bigr\}.
\]
Indeed, choose the $\ell$ smallest contacts for the terms that have reached their constant values. Therefore
\[
 m\Phi_j(t)=\max_{0\le\ell\le N}
 \left\{m\sum_{h:\,b_h>0}b_h\min\{t,c_{jh}\}
 -(m+N-\ell)t-B_j(\ell)\right\}.
\]
For $q_\ell\ge0$, maximizing the positive sum minus $(P-q_\ell)t$ gives $A_j(q_\ell)$: take $t$ at the $q_\ell$th smallest positive contact, or $t=0$ if $q_\ell=0$. Repeated contact values cause no change. Only the copies of the center have infinite contact, and there are at most $m$ of them, so these sums are finite for $q_\ell\le P-m$. If $q_\ell<0$, the expression is nonpositive and contributes nothing. This proves \eqref{eq:order-statistics}. When $P\le m$, the original expression satisfies $m\Phi_j(t)\le(P-m)t\le0$, proving $I=0$.

Finally, suppose $I>0$ and let $P_0,N_0$ count occurrences in $\mathcal C_{j_0}(t_0)$. The blowup inequalities give $P_0-N_0>m$ and $P_C-N_C\le m$ for each smaller set $C$ after the blowup. Set $\ell_0:=N-N_0$. Then
\[
 q_{\ell_0}=(P-P_0)+(P_0-N_0-m).
\]
For every positive center in the original set, the smallest positive contacts consist of all $P-P_0$ outside occurrences, followed by $P_0-N_0-m$ contacts equal to $t_0$. There are enough of the latter, since the smaller set containing the center has $P_C\le m+N_C\le m+N_0$ positive occurrences. The $\ell_0$ smallest negative contacts are precisely those outside the original set. Substituting in $m\Phi_j(t_0)$ gives $A_j(q_{\ell_0})-B_j(\ell_0)=mI$ for all $P_0>m$ positive occurrences, as required.
\end{proof}

\subsection*{Finite covers and coefficient ideals on the base}

\begin{prop}[An expression in terms of coherent ideals]\label{prop:ideals}
Let $\Theta$ be a $\Q$-divisor on $\Pj^n$, with coefficients at most one, and choose $q$ outside its support. Put $s=\deg\Theta^+$. Recall that $\Theta^\pm$ are its effective positive and negative parts, and $I_\Theta(\zeta)$ is the root-contact maximum of Notation~\ref{defn:contacts}. If $s\le1$, then $I_\Theta=0$. If $s>1$, there are finitely many triples $(\frakj_+,\frakj_-,d)$ of nonzero coherent ideals on $Z=\Pj^{n-1}$ and positive integers $d$, with $\frakj_\pm(d_\pm)$ globally generated for integers $d_\pm\ge0$, such that, simultaneously for all normalized divisorial $\zeta$,
\begin{equation}\label{eq:ideal-max}
 I_\Theta(\zeta)=\max\left\{0,\frac{\zeta(\frakj_+)-\zeta(\frakj_-)}d\right\}_{(\frakj_+,\frakj_-,d)},
 \qquad d_+/d\le s^2.
\end{equation}
Here $\zeta(\frakj)$ is the minimum value of local generators of the ideal. The integer $d$ is the normalizing denominator in the displayed quotient, while $d_\pm$ specify the twists $\frakj_\pm(d_\pm)=\frakj_\pm\otimes\cO_Z(d_\pm)$ that are globally generated.
No bound for $\deg\Theta^-$ is required.
\end{prop}
\begin{proof}
Lemma~\ref{lem:local-threshold-data} gives $I_\Theta=0$ when $s\le1$, so assume $s>1$. We now construct coefficient systems from the horizontal divisors. The local threshold calculation supplies their required valuation orders and multiplicities.

\smallskip\noindent\emph{Step 1: horizontal divisors as finite covers, and their collision loci.}
Write $\Theta=\sum_i b_i\Theta_i$, where the $\Theta_i$ are its prime components. Let $D_i$ be their strict transforms in $\Bl_qY$. The complement of the exceptional section in the projection bundle is the total space of $\cO_Z(1)$, and every $D_i$ lies in this complement. Choose homogeneous coordinates $[x:y_0:\cdots:y_{n-1}]$ with $q=[1:0:\cdots:0]$. The equation of $\Theta_i$ is homogeneous of degree $\deg\Theta_i$ and has nonzero leading coefficient in $x$; normalize that coefficient to one. On a trivializing chart of $\cO_Z(1)$, the equation is monic in the fiber coordinate. Its algebra has basis $1,x,\ldots,x^{\deg\Theta_i-1}$ over the base. Thus $D_i\to Z$ is finite flat of rank $\deg\Theta_i$, even where its fibers are singular or nonreduced.

Choose a positive integer $m$ with $mb_i\in\mathbb Z$ for every $i$. Form finite flat covers $p_\pm\colon\mathcal D_\pm\to Z$ by taking $mb_i$ disjoint copies of each $D_i$ with $b_i>0$, and $m|b_i|$ disjoint copies when $b_i<0$. These are copies of the schemes $D_i$, not nonreduced multiples of divisors. Their ranks are
\[
 P:=\sum_{b_i>0}mb_i\deg\Theta_i=ms,\qquad
 N:=\sum_{b_i<0}m|b_i|\deg\Theta_i.
\]
These ranks agree with the occurrence counts in Lemma~\ref{lem:local-threshold-data} for every geometric DVR of $Z$. We use its root labels and ordered sums $A_j,B_j$, with $m_j=mb_j$ at a positive root. When $N=0$, the negative cover is empty.

A point of $\mathcal D_+$ selects a positive boundary point in a projection fiber. The fiber products below then compare it with every positive or negative boundary point in that same fiber:
\[
\begin{tikzcd}
 \mathcal D_+\times_Z\mathcal D_\pm
 \arrow[r,"\operatorname{pr}_2"] \arrow[d,"\pi_\pm"']
 & \mathcal D_\pm \arrow[d,"p_\pm"]\\
 \mathcal D_+ \arrow[r,"p_+"'] & Z.
\end{tikzcd}
\]
The difference $\delta_\pm$ of the two fiber coordinates is a section of $(p_+\circ\pi_\pm)^*\cO_Z(1)$. Its zero locus is their collision locus. On the positive fiber product it is identically zero on the branches comparing a point with a copy of itself, so we retain it as a section rather than require a Cartier divisor everywhere. On a branch where it is not identically zero, its order of vanishing measures the contact of the two boundary points.

For instance, the three local sections $x=0$, $x=u$, $x=v$ over $\mathbb A^2_{u,v}$ meet together exactly over the ideal $(u,v)$. Such a collision need not occur along a divisor on the base. This is why we keep its common vanishing as a base ideal before passing to fixed divisors on a birational model.

\smallskip\noindent\emph{Step 2: systems on the positive cover and their norms on $Z$.}
Use the finite covers of Step~1 and fix an admissible $\ell$ from Lemma~\ref{lem:local-threshold-data}. The vector bundles $\mathcal E_\pm:=(\pi_\pm)_*\cO_{\mathcal D_+\times_Z\mathcal D_\pm}$ on $\mathcal D_+$ have ranks $P,N$. Multiplication by the squared collision section gives maps
\[
 \mathcal E_\pm\xrightarrow{\ \delta_\pm^2\ }
 \mathcal E_\pm\otimes p_+^*\cO_Z(2).
\]
Take the $q_\ell$th exterior power for the positive map and the $\ell$th exterior power for the negative map. Exterior powers here select products of collision factors: over a split point $\alpha_j$ their eigenvalues are precisely
\[
 \prod_{h\in J}(\alpha_j-\alpha_h)^2,
\]
with $J$ running over $q_\ell$ positive occurrences or $\ell$ negative occurrences, respectively. Positions in the occurrence list are selected separately, even when their roots coincide. For the zeroth exterior power we use the identity on $\cO$, whose sole eigenvalue is the empty product $1$; this convention also covers $N=0$.

The characteristic coefficients of these exterior-power maps give sections $Q_a^+,Q_a^-$ on $\mathcal D_+$. We choose their signs so that a characteristic polynomial of rank $M$ is $\sum_{a=0}^M(-1)^aQ_a\lambda^{M-a}$, with $Q_0=1$. Thus after splitting, for eigenvalues $\xi_1,\ldots,\xi_M$, including zeros and repetitions,
\[
 Q_a=\sum_{1\le i_1<\cdots<i_a\le M}\xi_{i_1}\cdots\xi_{i_a}.
\]
The characteristic coefficients $Q_a^+$ and $Q_a^-$ are sections of $p_+^*\cO_Z(2q_\ell a)$ and $p_+^*\cO_Z(2\ell a)$, respectively. The superscripts indicate the positive and negative occurrence lists, not the signs of these sections. These sections are defined on the cover without labeling its roots: changing a local basis conjugates the maps, and changing the fiber-coordinate trivialization gives exactly the indicated twists.

Choose $L$ divisible by every positive coefficient index in these finitely many characteristic polynomials, for all admissible $\ell$. Form universal combinations
\[
 U_\ell:=\sum_a T_a^+(Q_a^+)^{L/a},\qquad
 V_\ell:=\sum_a T_a^-(Q_a^-)^{L/a},
\]
where the sums run over the positive characteristic coefficient indices. Use independent formal parameter families for the two systems and for each $\ell$. Thus $U_\ell$ is a family of sections of $p_+^*\cO_Z(2Lq_\ell)$, and $V_\ell$ of $p_+^*\cO_Z(2L\ell)$. The parameters keep the entire families, so their minimum vanishing orders can be read simultaneously at every valuation.

To check these orders, fix any normalized divisorial valuation $\zeta$ of $K(Z)$ and extend it to a splitting field with the scaling of Lemma~\ref{lem:local-threshold-data}. For any finite eigenvalue list with some nonzero entry,
\[
 \min_{i:\,\xi_i\ne0}\zeta(\xi_i)
 =\min_{a\ge1:\,Q_a\ne0}\frac{\zeta(Q_a)}a.
\]
Every product in $Q_a$ has order at least $a$ times the left-hand side. If exactly $r$ entries have its minimum finite order, their product is the unique least-order term of $Q_r$, and cannot cancel. This proves equality, also with repeated entries. In the positive list some entries are zero, but $q_\ell\le P-m$ ensures a nonzero product at every positive center.

Let $U_{j\ell},V_{j\ell}$ be the evaluations of $U_\ell,V_\ell$ at the split center $\alpha_j$. These are the same homogeneous root polynomials obtained from the squared-product lists above, with total root degrees $2Lq_\ell$ and $2L\ell$. The Gauss extension $\zeta_G$ of the splitting-field valuation takes the minimum coefficient order in the formal parameters, so
\[
 \zeta_G(U_{j\ell})=2L A_j(q_\ell),\qquad
 \zeta_G(V_{j\ell})=2L B_j(\ell).
\]
All evaluations are nonzero. Their ratios therefore have the orders $2L(A_j(q_\ell)-B_j(\ell))$ occurring in \eqref{eq:order-statistics}.

We descend the families by the finite flat norm along $p_+$. In local trivializations of $\cO_Z(1)$, introduce an auxiliary variable $z$ and take
\begin{equation}\label{eq:encoding}
 \operatorname{Nm}_{p_+}(V_\ell z-U_\ell)
 =\prod_{j:\,b_j>0}(V_{j\ell}z-U_{j\ell})^{m_j}
 =\sum_{h=0}^P C_{\ell h}(y,T)z^{P-h}.
\end{equation}
The norm is the determinant of multiplication on the rank-$P$ algebra of the cover. It is consequently defined even over its singular or nonreduced fibers; the middle expression computes it on the split generic fiber. All centers use the same parameters because $U_\ell,V_\ell$ were constructed on the cover itself. The leading coefficient $\prod_j V_{j\ell}^{m_j}$ and constant coefficient $(-1)^P\prod_j U_{j\ell}^{m_j}$ are nonzero.

The norm coefficients glue on $Z$: a term in $C_{\ell h}$ contains $h$ factors of $U$ and $P-h$ factors of $V$. Equivalently, the auxiliary $z$ carries the difference $2L(q_\ell-\ell)$ of their line-bundle degrees. Hence every coefficient in the formal parameters $T$ is a section of $\cO_Z(d_{\ell h})$, where
\begin{equation}\label{eq:encoding-degree}
 d_{\ell h}=2L\bigl(hq_\ell+(P-h)\ell\bigr).
\end{equation}
For each nonzero $C_{\ell h}$, these sections span a linear system $\mathcal V_{\ell h}$ on $Z$. Let $\frakj_{\ell h}$ be its base ideal, as in Notation~\ref{defn:coefficient-ideals}. Thus its valuation is the common vanishing order of the coefficient family, and $\frakj_{\ell h}(d_{\ell h})$ is globally generated. These systems come from the original horizontal divisors and their collision sections; their construction is independent of $\zeta$.

\smallskip\noindent\emph{Step 3: fixed divisors recover the threshold loss.}
Take a common projective principalization $\rho\colon W\to Z$ of the finitely many base ideals, with $W$ smooth. Write
\[
 \frakj_{\ell h}\cO_W=\cO_W(-F_{\ell h}),\qquad
 \rho^*\mathcal V_{\ell h}=F_{\ell h}+\mathcal M_{\ell h},
\]
where $F_{\ell h}$ is effective and $\mathcal M_{\ell h}$ is base point free. The divisor $F_{\ell h}$ is the vanishing forced in every member; removing it leaves the pulled-back generating sections without common zeros.

Represent any divisorial $\zeta=\ord_E$ on a model dominating $W$. Some member of the moving system does not contain the center of $E$ on $W$, so its minimum order along $E$ is zero. Consequently
\[
 \zeta(\frakj_{\ell h})=\ord_E(F_{\ell h}).
\]
After locally trivializing the coefficient line bundle, this is also $\zeta_G(C_{\ell h})$. This identification holds for every divisorial valuation, including those represented only on higher models.

Divide the norm polynomial \eqref{eq:encoding} by its leading coefficient. Its roots are $U_{j\ell}/V_{j\ell}$, each counted $m_j$ times. For any monic polynomial of degree $P$ with nonzero roots, the constant coefficient has order equal to the sum of all root orders, whereas the coefficient indexed by $h$ has order at least the sum of the $h$ smallest root orders. Therefore
\[
 \frac{\zeta_G(C_{\ell P})-\zeta_G(C_{\ell h})}{P-h}
\]
is at most the largest root order. If $r$ roots, counted with repetitions, have that largest order, equality holds at $h=P-r$: the product of all roots of strictly smaller order is the unique least-order term in that coefficient. Here $r$ counts equal orders, without requiring the corresponding roots to be identical; when $h=0$ the product is empty.

If $I_\Theta(\zeta)>0$, Lemma~\ref{lem:local-threshold-data} supplies $\ell_0$ for which the largest root order is $2LmI_\Theta(\zeta)$, attained by more than $m$ positive occurrences. Thus $r>m$, and the coefficient giving equality lies among $0\le h\le P-m$. Every other retained quotient is at most $2LmI_\Theta(\zeta)$ by \eqref{eq:order-statistics}. If $I_\Theta(\zeta)=0$, that equation makes all root orders, and hence all retained quotients, nonpositive. The fixed divisors consequently give
\[
 I_\Theta(\ord_E)=\frac1{2Lm}
 \max\left\{0,
 \frac{\ord_E(F_{\ell P})-\ord_E(F_{\ell h})}{P-h}
 \right\}_{\ell,\ 0\le h\le P-m,\ C_{\ell h}\ne0}.
\]
Returning to their base ideals yields
\begin{equation}\label{eq:ideal-formula}
 I_\Theta(\zeta)=\frac1{2Lm}
 \max\left\{0,
 \frac{\zeta(\frakj_{\ell P})-\zeta(\frakj_{\ell h})}{P-h}
 \right\}_{\ell,\ 0\le h\le P-m,\ C_{\ell h}\ne0}.
\end{equation}

\smallskip\noindent\emph{Step 4: the degree bound.}
For each retained pair $\ell,h$, take
\[
 \frakj_+:=\frakj_{\ell P},\qquad
 \frakj_-:=\frakj_{\ell h},\qquad
 d:=2Lm(P-h).
\]
The positive system comes from the constant coefficient of the norm and has generating degree $d_+:=d_{\ell P}=2LPq_\ell$. Its comparison system has degree $d_-:=d_{\ell h}$. Both twists are globally generated. Since $q_\ell\le P-m\le P$ and $P-h\ge m$,
\[
 \frac{d_+}{d}
 =\frac{Pq_\ell}{m(P-h)}
 \le\frac{P^2}{m^2}=s^2.
\]
Thus \eqref{eq:ideal-formula} is precisely \eqref{eq:ideal-max}. The bound controls the degree of the positive generating system on the original base. The negative cover can have arbitrary rank: it affects which collisions we compare, but neither $P=ms$ nor this final bound depends on $\deg\Theta^-$.
\end{proof}

\begin{cor}[Continuity on monomial valuations]\label{cor:lipschitz}
Let $G_i(x,y_1,\dots,y_r)$ be monic in $x$, of degree $r_i$ in $x$ and total degree at most $d_i$, with squarefree product over $\C(y)$. Give its roots nonzero real coefficients $b_i\le1$, and put
$W=\sum|b_i|r_i$, $\gamma=\sum|b_i|d_i$. For nonnegative rational weight vectors $a,a'$ defining monomial valuations in the $y$ variables, the values $I(a)$ and $I(a')$ of Notation~\ref{defn:coefficient-ideals} satisfy
\[
 |I(a)-I(a')|\le4W\gamma\|a-a'\|_\infty\le4\gamma^2\|a-a'\|_\infty.
\]
They extend continuously to nonnegative real weight vectors, and vanish if $\sum_{b_i>0}b_ir_i\le1$.
\end{cor}
\begin{proof}
If $\sum_{b_i>0}b_ir_i\le1$, the positive terms in the root-contact expression are at most $t$, and the negative terms are nonpositive. Thus $I(a)=0$ for every rational weight vector, and all the assertions follow. For the remaining calculation, assume $\sum_{b_i>0}b_ir_i>1$.

For rational coefficients, apply the finite-cover construction in Proposition~\ref{prop:ideals} on the affine base $\mathbb A^r_y$, using the finite flat covers defined by the monic equations $G_i$. The exterior-power maps, norms, and coefficient ideals are defined in the same way; their valuation identities use only the local calculation of Lemma~\ref{lem:local-threshold-data}, so they remain valid without homogeneous equations. Set $P=m\sum_{b_i>0}b_ir_i$, $N=m\sum_{b_i<0}|b_i|r_i$. In \eqref{eq:encoding}, the individual degree in a root of $G_i$ is at most $4Lm^2W|b_i|$: factors where this root is $\alpha_j$ contribute at most $2Lm|b_i|\max\{P,N\}$, and factors where it is $\alpha_h$ contribute at most $2Lm|b_i|P$. The total root degree is at most $2Lm^2W\sum_{b_i>0}b_ir_i\le2Lm^2W^2$.

A symmetric polynomial of degree at most $N_i$ in each root of $G_i$ can be expressed using at most $N_i$ elementary symmetric factors in those roots, while preserving its total degree in the roots. To check this, take a lexicographically largest root monomial with exponents $\lambda_1\ge\cdots\ge\lambda_{r_i}$ and subtract the corresponding product of elementary functions, of exponents $\lambda_j-\lambda_{j+1}$. The number of factors is $\lambda_1\le N_i$; induction removes all monomials. The $a$th elementary symmetric polynomial in the roots of $G_i$ has $y$-degree at most $d_i-r_i+a$. Substitution therefore bounds every coefficient polynomial by
\[
 4Lm^2W\sum_i|b_i|(d_i-r_i)+2Lm^2W^2\le4Lm^2W\gamma.
\]
Specialize the finitely many auxiliary parameters preserving every nonzero $y$-monomial coefficient. This preserves all monomial valuations simultaneously. For each resulting polynomial $P(y)$, the function $a\mapsto\zeta_a(P)$ is $4Lm^2W\gamma$-Lipschitz with respect to $\|\cdot\|_\infty$. A difference is twice as Lipschitz, while every quotient in \eqref{eq:ideal-formula} is divided by at least $2Lm^2$. Taking a finite maximum gives $4W\gamma$. At the zero weight vector every nonzero coefficient polynomial has value zero, so the same formula gives $I(0)=0$ and the estimate includes this vector.

For real coefficients $b_i$, approximate them rationally, preserving signs and the upper bound one on positive coefficients. At each fixed nonzero rational weight vector in the base coordinates, the maximum is taken over $t$ up to the largest finite $c_{jh}$; if there are no finite contacts, $t=0$ already gives the maximum. On this fixed compact interval the expressions converge uniformly with the coefficients. Passing to the limit gives the asserted bound and independence of the chosen extension to the splitting field. The value at the zero weight vector is $I(0)=0$, and density gives the extension to all nonnegative real weight vectors. The vanishing assertion follows from the bound on the sum of positive coefficients, counted with multiplicity. The parameter specialization in this argument is used only for monomial continuity, not for the identity \eqref{eq:ideal-max} for all divisorial valuations.
\end{proof}

\section{Effective pairs on the base of the projection}\label{sec:descent}

When $a_\Theta(\zeta)\le1$, we construct an effective pair on a birational model of $Z$ whose log discrepancy at $\zeta$ is $a_\Theta(\zeta)$, with a bound on the degree of the pushforward of its boundary.

\begin{prop}[Construction of an effective pair on a birational model of $Z$]\label{prop:descent}
Let $n\ge3$, let $\Theta=\Theta^+-\Theta^-$ be a $\Q$-divisor on $Y=\Pj^n$, and assume $(Y,\Theta)$ is sub-$\eta$-lc and $\deg\Theta^+\le k$, where $k\ge1$ and $0<\eta\le1$. Choose $q\notin\Supp\Theta$, and put
\[
 \delta=d(k,\eta):=\frac{\eta}{2\lceil2(k+1)/\eta\rceil}.
\]
Then $a_\Theta(\zeta)\ge\delta$ for every normalized divisorial valuation $\zeta$ of $K(Z)$. If $\zeta$ is exceptional over $Z=\Pj^{n-1}$ and $a_\Theta(\zeta)\le1$, there is a projective birational morphism $g\colon V\to Z$ with an effective $\delta$-lc $\R$-boundary $C$, with $K_V+C$ $\R$-Cartier and $g$-nef, and with $\deg g_*C\le k^2$, such that $\zeta$ is represented by a prime divisor on $V$ contracted by $g$ and $A(\zeta,V,C)=a_\Theta(\zeta)$.
\end{prop}
\begin{proof}
Apply Proposition~\ref{prop:effective} to $(Y,\Theta)$ with no specified exceptional valuations. It gives a projective birational morphism $\widehat f\colon\widehat Y\to Y$ and an effective $\eta$-lc pair $(\widehat Y,\widehat C)$ with $\deg\widehat f_*\widehat C\le k$. Writing $\widehat G\ge0$ for the correction divisor in \eqref{eq:effective-comparison}, we have
\[
 A(v,\widehat Y,\widehat C)=A(v,Y,\Theta)-v(\widehat G)\le A(v,Y,\Theta).
\]
In Proposition~\ref{prop:ramification}, applied to this $\eta$-lc pair and the chosen point $q$, we have
\[
 \begin{aligned}
 S&=k+\max\{0,1-A(E_q,\widehat Y,\widehat C)\}\le k+1-\eta\le k+1,\\
 M&=\lceil2S/\eta\rceil\le\lceil2(k+1)/\eta\rceil.
 \end{aligned}
\]
Consequently, for every $v|_{K(Z)}=e\zeta$,
\[
 A(v,Y,\Theta)\ge A(v,\widehat Y,\widehat C)
 \ge\frac{\eta e}{2M}\ge\delta e.
\]
Taking the infimum proves $a_\Theta(\zeta)\ge\delta$. Notice also that $0<\delta\le\eta/8\le1/8$.

If $\deg\Theta^+\le1$, Proposition~\ref{prop:ideals} gives $a_\Theta(\zeta)=a_0(\zeta)$. For an exceptional valuation over the smooth $Z$, we have $a_0(\zeta)\ge2$, so there is no exceptional $\zeta$ with $a_\Theta(\zeta)\le1$. If $\deg\Theta^+>1$, combine \eqref{eq:direction-contact} and \eqref{eq:ideal-max} to write the valuative function $a_\Theta$ as the minimum of $a_0(\zeta)$ and finitely many functions
\begin{equation}\label{eq:signed-ideal-discrepancy}
 a_0(\zeta)-\alpha\zeta(\frakj_+)+\alpha\zeta(\frakj_-),\qquad \alpha=1/d.
\end{equation}
The minimum of $a_0(\zeta)$ and the expressions in \eqref{eq:signed-ideal-discrepancy} equals $a_\Theta(\zeta)\ge\delta$, as proved above. Hence each of these expressions is at least $\delta$ at every normalized divisorial valuation. At the specified exceptional $\zeta$ we have $a_0(\zeta)\ge2>a_\Theta(\zeta)$, so one of the ideal expressions attains $a_\Theta(\zeta)\le1$. Fix this expression. Proposition~\ref{prop:ideals} gives $d_+/d\le(\deg\Theta^+)^2$, so $\alpha=1/d$ and $\deg\Theta^+\le k$ imply $\alpha d_+\le k^2$.

Take a common log resolution $\rho\colon W\to Z$ of the two ideals, with $W$ smooth and projective, on which $\zeta$ is represented by a prime divisor. Write $\frakj_\pm\cO_W=\cO_W(-F_\pm)$ and define
\[
 K_W+C_W=\rho^*K_Z+\alpha F_+-\alpha F_-.
\]
The support of $C_W$ has simple normal crossings, and every coefficient of $C_W$ is at most $1-\delta$, by \eqref{eq:signed-ideal-discrepancy}. After removing $F_\pm$ from the pullbacks of the generating linear systems, the remaining systems are base point free. Choose a positive integer $r$ with $\alpha/r\le1-\delta$, which is possible since $\delta<1$, and apply Lemma~\ref{lem:free-members} successively to choose $r$ general members for $\frakj_+$ with remaining divisors $H_1,\ldots,H_r$. They have no common prime components with one another or with $\Supp C_W\cup\Supp F_+\cup\Supp F_-\cup\Exc(\rho)$, and their union with this divisor is SNC. Choose a general member for $\frakj_-$ with remaining divisor $H_-$, subject to the same conditions. All these members avoid containing the prime divisor of $W$ representing $\zeta$. If the generating degree is zero, the corresponding divisor on $Z$ is zero. For their corresponding divisors $D_i,D_-$ on $Z$, put
\[
 C_Z=\frac{\alpha}{r}\sum_{i=1}^rD_i-\alpha D_-.
\]
Its crepant pullback to $W$ has divisor $C_W+(\alpha/r)\sum H_i-\alpha H_-$. Dropping the final negative term lowers discrepancies, and the remaining positive SNC coefficients are at most $1-\delta$. Thus \eqref{eq:snc} proves $(Z,C_Z)$ is sub-$\delta$-lc for all valuations. The choices of the general members give
$A(\zeta,Z,C_Z)=a_\Theta(\zeta)$ exactly. Taking the positive part of $C_Z$ gives $\deg C_Z^+\le\alpha d_+\le k^2$.

Proposition~\ref{prop:effective}, now on $Z$ and with the specified exceptional valuation $\zeta$, gives the required pair $(V,C)$, the equality $A(\zeta,V,C)=a_\Theta(\zeta)$, and the bound on $\deg g_*C$. The degree of the negative part and the resolution used in this construction need no uniform bound.
\end{proof}

\begin{prop}[Comparison of discrepancies under projection]\label{prop:ambient-direction}
In Setup~\ref{setup:source}, choose $q$ in the isomorphism locus of $f$ and outside $\Supp\Delta$. Let $Z=\Pj^{n-1}$. There is an effective $\R$-divisor $T$ on $Z$, of degree at most $k^2$, such that for every $v|K(Z)=e\zeta\ne0$,
\begin{equation}\label{eq:ambient-general}
 (1+\eps)A(v,X,B)\ge\eps A(v,Y,0)+ea_0(\zeta)-e\zeta(T).
\end{equation}
For a prime divisor $E$ on $X$ contracted by $f$, with this restriction, one has
\begin{equation}\label{eq:ambient-direction}
 ea_0(\zeta)\le A(E,Y,0)
 \le (1+k^2/\eps)ea_0(\zeta)
 \le R(k,\eps)(1+k^2/\eps)a_0(\zeta).
\end{equation}
\end{prop}
\begin{proof}
Write $\Delta=\sum b_iD_i$ and normalize its degree-$d_i$ equations to be monic in the fiber coordinate. Define
\[
 T=\sum_i b_i^2\operatorname{Div}(\operatorname{disc}G_i)
   +2\sum_{i<j}b_ib_j\operatorname{Div}(\operatorname{res}(G_i,G_j)).
\]
A degree-one discriminant is the nonzero constant. The discriminant and resultant are nonzero homogeneous forms of degrees $d_i(d_i-1)$ and $d_id_j$. Thus $T$ is effective and $\deg T\le(\sum b_id_i)^2\le k^2$. In a splitting DVR its order is $Q=2\sum_{j<h}b_jb_hc_{jh}$.

For a valuation centered in the affine fiber chart, choose $\alpha_j$ as in Proposition~\ref{prop:contact}, with maximal normalized order $r_*$. Let $M(u)$ be the sum of the coefficients of the roots $\alpha_h$ satisfying $c_{jh}\ge u$, and let $N(u)$ be the sum of their squares. Then
\[
 \frac{v(\Delta)}e=\int_0^{r_*}M(u)\,du,\qquad
 Q\ge\int_0^{r_*}\bigl(M(u)^2-N(u)\bigr)\,du.
\]
The second inequality sums the valuations of differences of pairs of roots in these sets; all other contributions are nonnegative. Each root coefficient is at most $1-\eps$, since it equals the coefficient of the strict transform of the corresponding $D_i$ in $B$. Hence $N(u)\le(1-\eps)M(u)$ and
$1+M(u)^2-N(u)-(1+\eps)M(u)\ge(M(u)-1)^2\ge0$.
The local discrepancy estimate $A(v,\Pj^1_R,0)/e\ge1+r_*$ therefore yields
\[
 (1+\eps)A(v,\Pj^1_R,H)
 \ge\eps A(v,\Pj^1_R,0)+e-eQ.
\]
Adding the base Jacobian contribution and applying Lemma~\ref{lem:comparison} gives \eqref{eq:ambient-general}. At infinity the boundary order vanishes and $A(v,Y,0)\ge ea_0(\zeta)$, which gives the same inequality directly.

The local parameter calculation also gives $A(E,Y,0)\ge ea_0(\zeta)$. Effectivity gives $A(E,X,B)\le1$, whereas $ea_0(\zeta)\ge1$. Substitution in \eqref{eq:ambient-general} gives
\[
 \begin{aligned}
 0\le\eps\bigl(A(E,Y,0)-ea_0(\zeta)\bigr)
 &\le e\zeta(T)+(1+\eps)\bigl(1-ea_0(\zeta)\bigr)\\
 &\le e\zeta(T).
 \end{aligned}
\]
Lemma~\ref{lem:degree} bounds $\zeta(T)$ by $k^2a_0(\zeta)$, and Proposition~\ref{prop:ramification} bounds $e$ by $R(k,\eps)$.
\end{proof}

\section{The dimension induction}\label{sec:induction}

The construction in Proposition~\ref{prop:descent} produces a pair and a morphism satisfying the same hypotheses in one lower dimension.

\begin{thm}[A recursive bound for log discrepancies over projective space]\label{thm:ambient}
For rational $0<\eta\le1$ and integers $k\ge1$, define
\[
 \begin{gathered}
 M_1(k,\eta)=1,\\
 M_j(k,\eta)=\left\lceil R(k,\eta)(1+k^2/\eta)
 M_{j-1}(k^2,d(k,\eta))\right\rceil\quad(j\ge2),
 \end{gathered}
\]
where $R$ and $d$ are defined in Propositions~\ref{prop:ramification} and \ref{prop:descent}. For a real number $0<\eps\le1$, put $\tau=1/\lceil1/\eps\rceil$. Every $f$-exceptional prime divisor $E$ on $X$ in Setup~\ref{setup:source} satisfies
\[
 A(E,Y,0)\le M_n(k,\tau).
\]
\end{thm}
\begin{proof}
We first prove the assertion simultaneously for all integer degree bounds and rational $\eps=\eta>0$, for $\Q$-boundaries, by induction on dimension. On smooth $\Pj^1$ a divisorial valuation is the order of a closed point, with log discrepancy one. This is the meaning of the base value $M_1$.

In dimension $j\ge2$, choose $q$ in the isomorphism locus outside $\Supp\Delta$. Apply Proposition~\ref{prop:signed-target}, prescribing the valuation $E_q$ of the exceptional divisor of the blow-up at $q$. It gives a sub-$\eta$-lc sub-pair $(Y,\Theta)$, where $\Theta=\Delta-N$ is a $\Q$-divisor with $\deg\Theta^+\le k$, such that
\[
 A(E,Y,\Theta)=A(E,X,B)\le1
\]
for every $f$-exceptional prime divisor $E$ on $X$. Also $A(E_q,Y,\Theta)=j=A(E_q,Y,\Delta)$, so $\ord_{E_q}N=0$. Since $N$ is effective, $q\notin\Supp N$. Thus this same projection point is valid for $\Theta$.

Fix an $f$-exceptional prime divisor $E$ on $X$. Proposition~\ref{prop:ramification} shows that its restriction to $K(Z)$ is nontrivial, so write $\ord_E|_{K(Z)}=e\zeta$ with $e>0$. By the definition of $a_\Theta$ and the discrepancy equality just obtained,
\[
 a_\Theta(\zeta)\le\frac{A(E,Y,\Theta)}e
 =\frac{A(E,X,B)}e\le\frac1e\le1.
\]
If $\zeta$ is the valuation of a prime divisor on $Z=\Pj^{j-1}$, then $a_0(\zeta)=1$. This includes all such $\zeta$ for $j=2$. Otherwise $j\ge3$, and Proposition~\ref{prop:descent} constructs a projective birational morphism $g\colon T\to Z$ and an effective boundary $B_T$ such that $(T,B_T)$ is $d(k,\eta)$-lc, $K_T+B_T$ is nef over $Z$, $\deg g_*B_T\le k^2$, and $\zeta$ is the valuation of a $g$-exceptional prime divisor on $T$. Since $d(k,\eta)$ is rational, Lemma~\ref{lem:rational} replaces $B_T$ by a $\Q$-boundary on the same $T$, preserving these properties. The induction hypothesis applies and gives
$a_0(\zeta)\le M_{j-1}(k^2,d(k,\eta))$.
The same bound holds when $\zeta$ is represented by a prime divisor on $Z$, since then $a_0(\zeta)=1$ and $M_{j-1}(k^2,d(k,\eta))\ge1$.

Now apply Proposition~\ref{prop:ambient-direction} to the original effective pair $(X,B)$ and the morphism $f$, with $\eps=\eta$. Its hypotheses hold because $q$ lies in the isomorphism locus of $f$ and outside $\Supp\Delta$, and $E$ is an $f$-exceptional prime divisor on $X$. The upper bound in \eqref{eq:ambient-direction} gives
\[
 \begin{aligned}
 A(E,Y,0)
 &\le (1+k^2/\eta)ea_0(\zeta)\\
 &\le R(k,\eta)(1+k^2/\eta)a_0(\zeta)\\
 &\le R(k,\eta)(1+k^2/\eta)M_{j-1}(k^2,d(k,\eta))\\
 &\le M_j(k,\eta).
 \end{aligned}
\]
Here the second line uses $e\le R(k,\eta)$ from Proposition~\ref{prop:ramification}; the third uses the bound for $a_0(\zeta)$ just obtained in both cases; and the last uses the definition of $M_j$ as the ceiling of the preceding expression. This proves the rational assertion.

For an $\R$-boundary $B$, Lemma~\ref{lem:rational} with $\tau=1/\lceil1/\eps\rceil$ gives a $\tau$-lc $\Q$-boundary on the same $X$ and $f$. The divisors $E$ and their log discrepancies $A(E,Y,0)$ are unchanged. Applying the rational result proves the claim.
\end{proof}

\section{A uniform threshold and boundedness over projective space}\label{sec:boundedness}

We apply the bound for $A(E,Y,0)$ with $2k$ and $\eps/2$ in place of $k$ and $\eps$ to obtain a uniform positive number $\beta$ such that $(X,B+\beta f^*\Delta)$ is $\eps/2$-lc. We then deduce boundedness over the fixed $\Pj^n$ from boundedness of polarized pairs.

\begin{prop}[A uniform lower bound for the $\eps/2$-lc threshold]\label{prop:threshold}
Suppose $M\ge1$ bounds $A(E,Y,0)$ for every $f'$-exceptional prime divisor $E$ on $X'$, for all projective birational morphisms $f'\colon X'\to Y=\Pj^n$ and effective boundaries $B'$ such that $(X',B')$ is $\eps/2$-lc, $K_{X'}+B'$ is nef over $Y$, and $\deg f'_*B'\le2k$. Then every pair and morphism in Setup~\ref{setup:source} satisfies
\[
 (X,B+t f^*\Delta)\text{ is }\eps/2\text{-lc}
 \quad\text{for }0\le t\le\eps/(2kM).
\]
\end{prop}
\begin{proof}
If $\Delta=0$, there is nothing to prove. Otherwise resolve $X$, $B$ and $D=f^*\Delta$ on one smooth model $p\colon W\to X$, and write $K_W+B_W=p^*L$. Let $F_j$ be the components of the reduced SNC divisor $\Supp B_W\cup\Supp p^*D\cup\Exc(p)$, $a_j=A(F_j,X,B)\ge\eps$, and $d_j=\ord_{F_j}(p^*D)\ge0$. Set
\[
 t_0=\min_{d_j>0}\frac{a_j-\eps/2}{d_j}>0.
\]
The SNC inequality shows that $t_0=t_{\eps/2}(X,B;D)$ in Definition~\ref{defn:margin}, and $(X,B+tD)$ is $\eps/2$-lc for all $0\le t\le t_0$. Some $F$ attains the minimum, so
\begin{equation}\label{eq:threshold-witness}
 A(F,X,B+t_0D)=\eps/2,\qquad t_0\ord_F(\Delta)\ge\eps/2.
\end{equation}
Thus $t_0$ is computed by finitely many prime divisors on this one log resolution.

If $t_0>1$, the conclusion follows. Otherwise $(X,B+t_0D)$ is $\eps/2$-lc, its log canonical divisor is nef over $Y$, and its pushforward boundary has degree at most $2k$. If $F$ is a prime divisor on $Y$, then $A(F,Y,0)=1$. If it is exceptional over $Y$ and already lies on $X$, no extraction is needed. Otherwise extract it crepantly over $X$ by \cite[Corollary~1.4.3]{BCHM10}. The effective pair is klt, and the selected log discrepancy is at most one, so the prescribed extraction has effective crepant boundary; the condition concerning non-klt centers is vacuous. The resulting pair is $\eps/2$-lc, its log canonical divisor is nef over $Y$, and its pushforward boundary is unchanged. In the exceptional case, $F$ is now a prime divisor on a model satisfying the hypotheses; in either case, $A(F,Y,0)\le M$. Lemma~\ref{lem:degree} and \eqref{eq:threshold-witness} give
$t_0kM\ge t_0\ord_F(\Delta)\ge\eps/2$.
The extraction is used only to bound $A(F,Y,0)$; the asserted interval applies to the given pair on $X$.
\end{proof}

\begin{prop}[Boundedness from a uniform threshold]\label{prop:recovery}
Fix $n,k,\eps$ as in Setup~\ref{setup:source}. Suppose a collection of pairs and morphisms satisfying that setup admits a uniform rational $0<\beta\le1$ and a uniform $\eta>0$ such that $(X,B+\beta f^*\Delta)$ is $\eta$-lc. Then the morphisms $f\colon X\to\Pj^n$ form a bounded family over the fixed $\Pj^n$, and their $f$-exceptional prime divisors, with the reduced induced scheme structure on $X$, form a bounded family. More precisely, choose any integer $m\ge4/\beta$ and put
\[
 q:=\lceil m(1+\beta)k\rceil.
\]
Then there is an effective boundary $C_X$ on the same $X$ such that $(X,C_X)$ is $\eta_0$-lc for $\eta_0=\min\{\eta,1/2\}$, $K_X+C_X$ is $\R$-Cartier and nef over $Y$, $\deg f_*C_X=q/m$, and every nonzero coefficient of $C_X$ is at least $1/m$. Every $f$-exceptional prime divisor on $X$ has positive coefficient in $C_X$.
\end{prop}
\begin{proof}
Put $Q=(1+\beta)\Delta$ and $\frakj=\cJ(Y,mQ)$. The multiplier-ideal estimates and generation theorem give
\begin{equation}\label{eq:multiplier}
 mv(Q)-A(v,Y,0)<v(\frakj)\le mv(Q),\qquad
 \frakj(qH)\text{ globally generated}.
\end{equation}
We recall their scope, because $Q$ need not be lc. For an effective $\Q$-divisor $Q$, the strict lower estimate is the log-resolution round-down definition. Repeated subadditivity \cite[Theorem~2.3]{DEL00Real}, after clearing denominators by $a$, gives
$\cO_Y(-amQ)=\cJ(amQ)\subseteq\cJ(mQ)^a$, hence the upper estimate. The generation assertion follows from Nadel vanishing and Castelnuovo--Mumford regularity, as in \cite[Corollary~1.5]{DEL00Real}, since
$qH-K_Y-nH-mQ\sim_\R(q+1-m(1+\beta)\deg\Delta)H$ is ample. For an $\R$-divisor $Q$, approximate its coefficients rationally from above. On a fixed log resolution all the finitely many round-down coefficients stabilize, including those with integral limits. The ample inequality persists. This proves generation and the upper estimate by passage to the limit; the lower estimate follows on a resolution also representing the specified valuation. No lc premise for $Q$ enters these arguments.

Take a sufficiently general effective divisor $D$ in the linear system defined by the global sections of $\frakj(qH)$ and define
\begin{equation}\label{eq:floor-replacement}
 C_X=B+f^*(D/m-\Delta).
\end{equation}
Here $D$ is an effective divisor with integral coefficients; fixed components may occur. The divisor $K_X+C_X$ is $\R$-Cartier and retains the numerical class of $K_X+B$ over $Y$. Choose a common log resolution and principalization $g\colon W\to Y$ factoring through $p\colon W\to X$, so that $\Supp B_W\cup\Supp g^*\Delta\cup\Supp F_{\frakj}\cup\Exc(g)$ is SNC, where $\frakj\cO_W=\cO_W(-F_{\frakj})$ and $K_W+B_W=p^*(K_X+B)$. The total transform of $D$ is $F_{\frakj}+M_D$. After subtracting $F_{\frakj}$, the pullbacks of the sections of $\frakj(qH)$ define a base point free linear system, and $M_D$ is a general member of this system. By Lemma~\ref{lem:free-members}, $M_D$ is reduced and smooth, has no component in common with this fixed divisor, and its union with that divisor is SNC. The divisor $C_W$ defined by $K_W+C_W=p^*(K_X+C_X)$ is
\[
 B_W-g^*\Delta+F_{\frakj}/m+M_D/m.
\]
For each component of this fixed reduced divisor, the upper estimate in \eqref{eq:multiplier} bounds its coefficient in $C_W$ above by its coefficient in $B_W+\beta g^*\Delta$, hence by $1-\eta$. Each component of $M_D$ has coefficient $1/m\le1/4$ in $C_W$. Formula \eqref{eq:snc} gives $A(v,X,C_X)\ge\eta_0$ for every divisorial valuation $v$.

Effectivity on $X$ needs a separate argument. For an $f$-exceptional prime divisor $E$ on $X$, write $a_E=A(E,Y,0)\ge2$ and $b_E=\operatorname{coeff}_EB\ge0$. From $F\ge0$ and its coefficient identity,
$\ord_E(\Delta)\ge a_E-1+b_E$. The lower multiplier estimate therefore gives
\begin{align*}
 \operatorname{coeff}_EC_X
 &>b_E+\beta \ord_E(\Delta)-a_E/m\\
 &\ge b_E+\beta(a_E-1)-a_E/m
 \ge b_E+\beta/2.
\end{align*}
The last inequality follows from $a_E\ge2$ and $m\ge4/\beta$. At a nonexceptional prime the strict transform of $\Delta$ cancels in \eqref{eq:floor-replacement}, leaving its coefficient in $D/m$, either zero or at least $1/m$. Since $\beta/2\ge2/m$, this proves effectivity and the asserted lower bound for the nonzero coefficients. The pushforward is $D/m$.

We now apply the polarized boundedness theorem, including its numerical hypotheses. Put $H_X=f^*H$. Choose $r$ general pullback hyperplanes and give each coefficient $a=2n/r\le1-\eta_0$. On the chosen resolution their union with $\Supp C_W$ is SNC by successive applications of Lemma~\ref{lem:free-members}; they have no common prime components with one another or with $\Supp C_W$. With $G$ their weighted sum and $\Theta_X=C_X+G$, the SNC calculation shows that $(X,\Theta_X)$ is $\eta_0$-lc and that every nonzero coefficient of $\Theta_X$ is at least $\min\{1/m,a\}$. The cone theorem as in Proposition~\ref{prop:signed-target} shows that $K_X+\Theta_X$, which is $\R$-linearly equivalent to $K_X+C_X+2nH_X$, is globally nef. The divisor $F_C=f^*(K_Y+f_*C_X)-(K_X+C_X)$ is effective by Lemma~\ref{lem:comparison}. For a further general hyperplane pullback $H_*$,
\[
 \vol(K_X+\Theta_X+H_*)
 =\vol((q/m+n)H_X-F_C)\le(q/m+n)^n.
\]
The divisor $H_*$ is nef and big. It is an effective Cartier divisor whose underlying closed subscheme is integral: a general target hyperplane contains no exceptional center, so its pullback has one prime component of multiplicity one; an effective Cartier divisor on a noetherian normal scheme is $S_1$ by \cite[Tag~0B3S]{StacksFoundations}, and generic reducedness together with $S_1$ implies reducedness by \cite[Tag~0344]{StacksFoundations}. Thus \cite[Theorem~1.5]{BirkarPolarised} bounds the couples $(X,\Supp(\Theta_X+H_*))$.

To recover the morphism, retain $H_*$ as a component of the reduced divisor in these bounded families of couples. Reduced irreducible components of fibers of a fixed projective finite-type family are parameterized by finitely many finite-type families: take the finitely many geometric generic components after finite covers, spread their reduced closures, shrink to flat geometrically integral fibers, and repeat on the lower-dimensional complement. The generic fiber component and geometric integrality statements used here are recorded in \cite[Tags~054Y, 0574, 0553]{StacksComponents}. This is noetherian induction, and therefore terminates. The construction retains the embedded reduced component, without normalization.

Stratify the resulting families of varieties and divisors to be flat, and write $h\colon\mathcal X\to S$ for one such family. Retain the open locus where the chosen component is relative effective Cartier; fiberwise Cartierness gives this open neighborhood by \cite[Tag~062Y]{StacksMarking}, and properness permits deleting the image of its complement. Denote the chosen relative effective Cartier divisor by $\mathcal H_*\subset\mathcal X$. Its line bundle $\mathcal M=\cO_{\mathcal X}(\mathcal H_*)$ restricts to $\cO_X(H_X)$ on the desired fibers. After further stratification, cohomology and base change give finitely many strata on which $h_*\mathcal M$ is locally free and commutes with base change; see \cite[Tag~07VJ]{StacksMarking}. Retain the strata on which its rank is $n+1$ and the open locus where the evaluation map $h^*h_*\mathcal M\to\mathcal M$ is surjective. Every desired fiber remains. Indeed $f_*\cO_X=\cO_Y$ by \cite[Tag~0AY8]{StacksFoundations}: $f$ is proper birational, $X$ is integral, and $Y$ is normal and integral. Trivializing a line bundle on $Y$ and gluing gives $f_*f^*\cO_Y(1)=\cO_Y(1)\otimes f_*\cO_X$. Hence the pulled coordinate sections form a complete basis of $H^0(X,H_X)$, of dimension $n+1$, and generate $H_X$. Pass to the scheme over $S$ parameterizing ordered bases of $h_*\mathcal M$; locally on $S$, it is $S\times\operatorname{GL}_{n+1}$. Its universal ordered basis defines a morphism from the pulled-back family to the fixed $\Pj^n$. The basis pulled back from the target coordinates recovers exactly the given $f$. These parameter spaces are of finite type, proving boundedness of the morphisms over the fixed $\Pj^n$.

Finally, every $f$-exceptional prime divisor on $X$ has positive coefficient in $C_X$, so it is among the reduced components of the divisors in the bounded families of couples. The component construction just given bounds its reduced induced scheme structure on $X$, including all nonnormal fibers. This proves the assertion about exceptional divisors as schemes.
\end{proof}

\begin{proof}[Proof of Theorem~\ref{thm:main}]
For $n=1$, a projective birational morphism of normal integral curves to $\Pj^1$ is proper and quasi-finite, hence finite, and is an isomorphism by normality of the target. Assume $n\ge2$. Theorem~\ref{thm:ambient} applied with $2k$ and $\eps/2$ in place of $k$ and $\eps$ gives
\[
 M=M_n\left(2k,\frac1{\lceil2/\eps\rceil}\right).
\]
Proposition~\ref{prop:threshold} shows that $(X,B+\beta f^*\Delta)$ is $\eps/2$-lc for the positive rational choice
$\beta=1/\lceil2kM/\eps\rceil$. Proposition~\ref{prop:recovery} proves both boundedness assertions for the given varieties $X$ and morphisms $f$. Finally Lemma~\ref{lem:comparison} gives
$A(v,X,B)\ge A(v,Y,0)-v(\Delta)$, while $(X,B+\beta f^*\Delta)$ being $\eps/2$-lc gives
$A(v,X,B)\ge\eps/2+\beta v(\Delta)$. Eliminating $v(\Delta)$ gives \eqref{eq:main-comparison}.
\end{proof}

\begin{proof}[Proof of Corollary~\ref{cor:bounded-divisor-threshold}]
First assume $n\ge2$ and take $t=\beta/(2(1+\beta)\ell)$, where $\beta$ is the constant in Theorem~\ref{thm:main}. The discrepancy comparison~\eqref{eq:main-comparison} implies
\[
 A(v,X,B)\ge\frac{\beta}{1+\beta}A(v,Y,0)
\]
for every normalized divisorial valuation $v$. Lemma~\ref{lem:degree} applies to any effective $D$ of degree at most $\ell$, and hence
\[
 t\,v(D)\le\frac{\beta}{2(1+\beta)}A(v,Y,0)
 \le\frac12 A(v,X,B).
\]
Every $\R$-divisor on $Y=\Pj^n$ is $\R$-Cartier, so its pullback is defined, and $v(f^*D)=v(D)$. Thus
\[
 A(v,X,B+t f^*D)
 =A(v,X,B)-t\,v(D)
 \ge\frac12 A(v,X,B)\ge\eps/2.
\]
This proves the assertion simultaneously for all such $D$. If $D=0$, its threshold is $\infty$.

For $n=1$, the morphism $f$ is an isomorphism. Every normalized divisorial valuation is the order of a closed point, where the coefficient of $D$ is at most $\ell$. Consequently $t=\eps/(2\ell)$ leaves every log discrepancy at least $\eps/2$.

For the polarized formulation, write $H=\cO_Y(1)$ and $A\sim aH$ with $a\in\mathbb Z_{>0}$. The pseudo-effective cone of $\Pj^n$ is the nonnegative ray generated by $H$, so $A-D$ is pseudo-effective precisely when $\deg D\le a$. Since $a^n=A^n\le r$, taking $\ell=\lceil r^{1/n}\rceil$ gives the required uniform constant. In particular, an ample $A-D$ also satisfies the hypothesis.
\end{proof}

\begin{cor}[Log resolutions in bounded families]\label{cor:references}
For fixed $n,k,\eps$, there are finitely many log smooth projective families of couples $(S,\Sigma)$ with projective birational morphisms $S\to\Pj^n$ on their fibers such that, for every morphism $f\colon X\to\Pj^n$ in Setup~\ref{setup:source}, the valuations of all $f$-exceptional prime divisors on $X$ are represented simultaneously by distinct components of $\Sigma$ on a single fiber. The number of these exceptional divisors is uniformly bounded; with their reduced induced scheme structures on $X$, they are in particular birationally bounded.
\end{cor}
\begin{proof}
Start with the bounded families of morphisms just proved, together with the reduced exceptional divisors. On each parameter component resolve the geometric generic fiber and these divisors, after a finite extension over which their geometric components are defined. Spread this resolution and shrink so its fibers are smooth and the union of the strict transforms of the chosen divisors and the exceptional divisors of the map to $\Pj^n$ is relative SNC. Repeat on the complement. Noetherian induction yields finitely many such families of resolutions. A resolution retains the strict transform of every $f$-exceptional prime divisor on $X$, so a single fiber represents all of them. The number of components of $\Sigma$ is bounded in these finite-type families.
\end{proof}
\section{Quantitative consequences and criteria}\label{sec:applications}

The recursive constants admit explicit power bounds. We also relate boundedness over the fixed $\Pj^n$ to a uniform positive lower bound for an $\eta$-lc threshold and to an inequality between $A(v,X,B)$ and $A(v,Y,0)$.

\begin{proof}[Proof of Theorem~\ref{thm:explicit}]
In the recursion of Theorem~\ref{thm:ambient}, elementary rounding inequalities give
\[
 R(k,\eta)\le6k\eta^{-2},\qquad
 1+k^2/\eta\le2k^2\eta^{-1},\qquad
 d(k,\eta)\ge\eta^2/(12k).
\]
Since the expression inside the outer ceiling is at least one,
$M_j(k,\eta)\le24k^3\eta^{-3}M_{j-1}(k^2,d(k,\eta))$.
Put $a_1=b_1=0$ and $C_1=1$. The recurrences
\[
 a_j=3+2a_{j-1}+b_{j-1},\quad b_j=3+2b_{j-1},\quad
 C_j=24C_{j-1}12^{b_{j-1}}
\]
then prove $M_j(k,\eta)\le C_jk^{a_j}\eta^{-b_j}$ by induction. Their solutions for $j\ge2$ are exactly the constants in the theorem. Since $\eps/2\le1/\lceil1/\eps\rceil\le\eps$, the estimate for $A(E,Y,0)$ follows.

For the lower bound for $A(v,X,B)$, set $\tau=1/\lceil2/\eps\rceil$ and $M=M_n(2k,\tau)$. Then $\eps/3\le\tau\le\eps/2$ and
\[
 M\le C_n2^{a_n}3^{b_n}k^{a_n}\eps^{-b_n}.
\]
Use Proposition~\ref{prop:threshold} with $\beta=\tau/(kM)$. It is rational, at most one, and at most $\eps/(2kM)$. Formula \eqref{eq:main-comparison} gives
\[
 A(v,X,B)\ge\frac\beta{1+\beta}A(v,Y,0)
 \ge\frac\eps{6kM}A(v,Y,0),
\]
which is the asserted estimate.
\end{proof}

\begin{prop}[Equivalent boundedness criteria]\label{prop:criteria}
For any collection of pairs and morphisms satisfying Setup~\ref{setup:source}, the following are equivalent:
\begin{enumerate}
\item the morphisms $f\colon X\to\Pj^n$ form a bounded family over the fixed $\Pj^n$;
\item there are uniform rational $\beta>0$ and $\eta>0$ such that $(X,B+\beta f^*\Delta)$ is $\eta$-lc;
\item there is a uniform $\delta>0$ with $A(v,X,B)\ge\delta A(v,Y,0)$ for every divisorial valuation $v$ and every pair and morphism in the collection.
\end{enumerate}
In (2), the stronger inequality $(1+\beta)A(v,X,B)\ge\beta A(v,Y,0)+\eta$ holds. In (3), any rational $0<\beta\le\min\{1,\delta/(2k)\}$ gives (2) with $\eta=\delta/2$.

A sufficient local criterion for (3) is the following. Fix $0<c<1$ and $\theta>0$. For each pair and morphism in the collection and each $v$, suppose there is a projective birational morphism $\rho\colon S\to Y$ with $S$ smooth, together with a reduced SNC divisor $\Sigma$ containing every $\rho$-exceptional prime divisor. Write $\Delta_S$ for the strict transform of $\Delta$, and put
\[
 R=\sum_{D\not\subseteq\Supp\Sigma}\operatorname{coeff}_D(\Delta_S)D,
\]
where the sum runs over the prime components of $\Delta_S$ not contained in $\Supp\Sigma$. Suppose $(S,\Sigma+R/c)$ is lc near the center of $v$, and every component $F_i$ of $\Sigma$ through that center satisfies
$A(F_i,X,B)\ge\theta A(F_i,Y,0)$.
Then (3) holds with $\delta=\min\{1-c,\theta\}$. The models $S\to Y$ may depend on $v$ and need not form a bounded family.
\end{prop}
\begin{proof}
Proposition~\ref{prop:recovery} proves (2)$\Rightarrow$(1), after decreasing $\beta$ to at most one. Combining the $\eta$-lc condition in (2) with Lemma~\ref{lem:comparison} proves the stronger displayed inequality and (2)$\Rightarrow$(3). Conversely Lemma~\ref{lem:degree} gives $v(\Delta)\le kA(v,Y,0)$. Under (3), subtraction of $\beta v(\Delta)$ gives at least $(\delta/2)A(v,Y,0)\ge\delta/2$, proving (2).

For completeness, (1)$\Rightarrow$(2) can be proved independently of Theorem~\ref{thm:main}. On the bounded families of morphisms choose relative very ample divisors $A_0$ after finite stratification. Their top degree, the numbers $f^*H\cdot A_0^{n-1}$, and the $A_0$-degrees of reduced exceptional divisors are bounded. The last assertion follows by resolving the geometric generic morphism and spreading its finite exceptional support, followed by noetherian induction, as in Corollary~\ref{cor:references}. Let $E_1,\ldots,E_s$ be the $f$-exceptional prime divisors on $X$. Since
$B\le f^*\Delta+\sum_{i=1}^s E_i$, both $B$ and $T=f^*\Delta$ have bounded $A_0$-degree.

A common integral multiple $A$ of $A_0$ can consequently be chosen so $A-B$ and $A-T$ are pseudo-effective. To justify this for Weil boundaries, embed $X$ by $A_0$ and successively project from points outside the current image until reaching $\Pj^n$. Each projection is finite by \cite[Tag~0B1P]{StacksFoundations}; the integral image keeps dimension $n$, so the final map is surjective. The resulting $n+1$ linear sections have no common zero and pull back $\cO_{\Pj^n}(1)$ to exactly $A_0$. For a prime divisor $D$, its effective cycle pushforward is a hypersurface of degree $A_0^{n-1}\cdot D$. Its Cartier pullback contains $D$: the pullback multiplicity is the ramification index times the positive pushforward residue degree. Thus this pullback is an effective Cartier divisor $D'\ge D$ with $D'\sim(A_0^{n-1}\cdot D)A_0$. Apply this construction to each prime component of $B$ and $T$, multiply by its coefficient, and choose a uniform integer multiple of $A_0$ using their bounded $A_0$-degrees. The top degree of the resulting very ample $A$ is bounded.

Apply \cite[Theorem~1.8]{Birkar16SourceMargin} to $(X,B)$, this $A$, and the effective $\R$-Cartier divisor $T$. Its hypotheses are precisely fixed $n$ and $\eps>0$, bounded $A^n$, and pseudo-effectivity of $A-B$ and $A-T$. It gives a uniform positive lower bound for $t_0(X,B;T')$ as $T'$ ranges over effective $\R$-divisors with $T'\sim_\R T$. Taking $T'=T$ gives the required bound. Choose rational $\beta$ at most half that bound and at most one. Convexity with the given $\eps$-lc boundary shows that $(X,B+\beta T)$ is $\eps/2$-lc.

For the local criterion write $\alpha_i=v(F_i)$, $a_i=A(F_i,X,B)$, $u_i=A(F_i,Y,0)$, and $j=A(v,S,\Sigma)$. Lemma~\ref{lem:comparison} gives
\[
 A(v,X,B)\ge j+\sum_i\alpha_i a_i-v(R).
\]
The lc hypothesis on $(S,\Sigma+R/c)$ gives $v(R)\le cj$, and the assumed comparison at the components of $\Sigma$ gives $a_i\ge\theta u_i$. Because $\Sigma$ contains the support of $K_S-\rho^*K_Y$, the canonical divisor formula gives
$A(v,Y,0)=j+\sum_i\alpha_i u_i$.
Thus $A(v,X,B)\ge(1-c)j+\theta\sum_i\alpha_i u_i\ge\min\{1-c,\theta\}A(v,Y,0)$.
\end{proof}

\begin{rem}\label{rem:parent-cap}
In the local criterion it suffices, for each pair and morphism in the collection, to choose one projective birational morphism $\rho\colon S\to Y$ and one reduced SNC divisor $\Sigma$ as above such that $(S,\Sigma+R/c)$ is lc everywhere and $A(F,Y,0)\le M$ for every $\rho$-exceptional prime divisor $F$, for a uniform $M\ge1$. Every component $F_i$ of $\Sigma$ then satisfies $A(F_i,X,B)\ge\eps$ and $A(F_i,Y,0)\le M$, since $A(F_i,Y,0)=1$ when $F_i$ is not $\rho$-exceptional. Thus one may take $\theta=\eps/M$. This implication does not require a bound on the number of components of $\Sigma$ or on the number of birational modifications used to construct $S$.
\end{rem}

\begin{prop}[Non-lc pushforward pairs with coefficients tending to zero]\label{prop:examples}
For every $n\ge2$ and every fixed $0<\eps<1$, there are pairs and morphisms satisfying the hypotheses of Theorem~\ref{thm:main} for which $(Y,f_*B)$ is not lc, $\deg f_*B=3$, and the positive coefficients of $f_*B$ tend to zero.
\end{prop}
\begin{proof}
Blow up a codimension-two linear subspace $L\subset\Pj^n$, with exceptional divisor $E$. The pencil of hyperplanes through $L$ induces a morphism $X\to\Pj^1$. The strict transforms of distinct pencil members $H_1,\dots,H_m$ are disjoint smooth fibers. Set
\[
 \Delta_m=\frac3m\sum_iH_i,\qquad B_m=\frac3m\sum_iH_i^{\mathrm{str}},
 \qquad 3/m\le1-\eps.
\]
The SNC criterion shows that $(X,B_m)$ is $\eps$-lc. Since $f^*\Delta_m=B_m+3E$ and $K_X=f^*K_Y+E$,
$K_X+B_m=f^*(K_Y+\Delta_m)-2E$ is relatively ample. The log discrepancy $A(E,Y,\Delta_m)$ is $2-3=-1$. The coefficients tend to zero and $\deg\Delta_m=3$ for every $m$.
\end{proof}

\section*{Use of artificial intelligence}
The author developed this paper with the assistance of artificial intelligence tools, including GPT and DeepSeek. He began working on the problem after hearing Professor Caucher Birkar state the conjecture at a conference in Sanya approximately two years before the completion of this paper. The main obstacle was to control the degree of the boundary produced in the induction on dimension. AI-assisted discussions helped the author overcome this difficulty through the degree estimate in Proposition~\ref{prop:ideals}: the normalized generating degree of the positive ideal is bounded independently of the degree of the negative part of the auxiliary divisor. This degree control, which makes the induction possible, was the principal mathematical contribution of the AI assistance. The author takes responsibility for the mathematical content and the final presentation of the paper.

\end{document}